\documentclass[a4paper,12pt]{amsart}
\usepackage{fullpage}

\usepackage{amsmath, amssymb, amsthm,enumitem,bm,pict2e,xcolor,tikz-cd,indentfirst,xparse,ytableau,stackengine}
\usepackage[foot]{amsaddr}
\usepackage[backref=page]{hyperref}

\hypersetup{
    colorlinks=true,
    linkcolor=blue,
    filecolor=magenta,      
    urlcolor=cyan,
}

\theoremstyle{definition}
\newtheorem{theorem}{Theorem}
\newtheorem*{theorem*}{Theorem}
\numberwithin{theorem}{subsection}
\newtheorem{definition}[theorem]{Definition}
\newtheorem*{definition*}{Definition}
\newtheorem{proposition}[theorem]{Proposition}
\newtheorem{lemma}[theorem]{Lemma}
\newtheorem{remark}[theorem]{Remark}
\newtheorem{example}[theorem]{Example}
\newtheorem{cor}[theorem]{Corollary}

\newtheorem{innercustomthm}{Theorem}
\newenvironment{customthm}[1]
  {\renewcommand\theinnercustomthm{#1}\innercustomthm}
  {\endinnercustomthm}

\DeclareMathOperator{\sgn}{sgn}

\DeclareMathOperator{\grad}{grad}
\DeclareMathOperator{\spann}{span}

\DeclareMathOperator{\id}{id}

\DeclareMathOperator{\gr}{gr}

\DeclareMathOperator{\initial}{in}

\DeclareMathOperator{\wt}{wt}

\DeclareMathOperator{\MultiProj}{MultiProj}
\DeclareMathOperator{\Proj}{Proj}
\DeclareMathOperator{\conv}{conv}
\DeclareMathOperator{\grdim}{grdim}

\newcommand{\la}{\lambda}
\newcommand{\om}{\omega}
\newcommand{\cU}{\mathcal U}
\newcommand{\cL}{\mathcal L}
\newcommand{\cJ}{\mathcal J}
\newcommand{\cO}{\mathcal O}

\newcommand{\cP}{\mathcal P}
\newcommand{\cR}{\mathcal R}
\newcommand{\cQ}{\mathcal Q}
\newcommand{\fg}{\mathfrak{g}}
\newcommand{\fsl}{\mathfrak{sl}}
\newcommand{\fgl}{\mathfrak{gl}}
\newcommand{\fn}{\mathfrak{n}}
\newcommand{\fh}{\mathfrak{h}}

\newcommand{\bC}{\mathbb{C}}
\newcommand{\bP}{\mathbb{P}}
\newcommand{\bR}{\mathbb{R}}
\newcommand{\bZ}{\mathbb{Z}}

\newcommand{\one}{\mathbf 1}
\newcommand{\rA}{\mathrm A}
\newcommand{\rB}{\mathrm B}
\newcommand{\rC}{\mathrm C}
\newcommand{\al}{\langle}
\newcommand{\ar}{\rangle}
\newcommand{\bs}{\backslash}

\newcommand\ledot{\mathrel{\ensurestackMath{%
  \stackengine{-.5ex}{\lessdot}{-}{U}{c}{F}{F}{S}}}}

\DeclareFontFamily{U}{mathx}{\hyphenchar\font45}
\DeclareFontShape{U}{mathx}{m}{n}{
      <5> <6> <7> <8> <9> <10>
      <10.95> <12> <14.4> <17.28> <20.74> <24.88>
      mathx10
      }{}
\DeclareSymbolFont{mathx}{U}{mathx}{m}{n}
\DeclareMathSymbol{\bigtimes}{1}{mathx}{"91}

\mathcode`l="8000
\begingroup
\makeatletter
\lccode`\~=`\l
\DeclareMathSymbol{\lsb@l}{\mathalpha}{letters}{`l}
\lowercase{\gdef~{\ifnum\the\mathgroup=\m@ne \ell \else \lsb@l \fi}}%
\endgroup

\mathchardef\newbracket=\mathcode`)
\mathcode`)="8000
\begingroup
\catcode`) \active
\gdef){\nolinebreak\newbracket}
\endgroup

\mathchardef\newcomma=\mathcode`,
\mathcode`,="8000
\begingroup
\catcode`, \active
\gdef,{\nolinebreak\newcomma}
\endgroup

\setlist[enumerate,1]{leftmargin=8mm}
\setlist[enumerate,2]{leftmargin=8mm}
\setlist[itemize]{leftmargin=8mm}

\title{Poset polytopes and pipe dreams: types C and B}

\author{Ievgen Makedonskyi$^*$}
\address{$^*$Beijing Institute of Mathematical Sciences and Applications (BIMSA)}

\author{Igor Makhlin$^{**}$}
\address{$^{**}$Technische Universit\"at Berlin, \href{mailto://iymakhlin@gmail.com}{iymakhlin@gmail.com}}

\begin{document}

\maketitle

\begin{abstract}
The first part of this paper concerns type C. We present new explicitly defined families of algebro-combinatorial structures of three kinds: combinatorial bases in representations, Newton--Okounkov bodies of flag varieties and toric degenerations of flag varieties. All three families are parametrized by the same family of polytopes: the marked chain-order polytopes of Fang and Fourier which interpolate between the type C Gelfand--Tsetlin and FFLV polytopes. Thus, in each case the obtained structures interpolate between the well-known bases, Newton--Okounkov bodies or degenerations associated with the latter two polytopes. We then obtain similar results for type B after introducing a new family of poset polytopes to be considered in place of marked chain-order polytopes. In both types our constructions and proofs rely crucially on a combinatorial connection between poset polytopes and pipe dreams. 
\end{abstract}

{\footnotesize\tableofcontents}

\section*{Introduction}

In modern Lie theory several kinds of structures have established themselves as useful tools providing combinatorial insights into the geometry of flag varieties and into the representation theory of the respective algebras. These tools include toric degenerations of flag varieties (\cite{GL,Ca,AB,GHKK,SoS}), their Newton--Okounkov bodies (\cite{Ka,FFL3,Ki17,FuHi,Fu}) and various combinatorial bases in representations, in particular, PBW-monomial bases (\cite{FFL1,FaFL2,MtypeB,Gor19,ChL}). 

Furthermore, it is somewhat typical for a structure of one of these kinds to be accompanied by structures of the other kinds. This phenomenon is observed in a range of works in which an algebro-combinatorial datum of a certain form is shown to provide several of the above structures. Perhaps, the first and best-known example are the string polytopes due to~\cite{L,BZ01} which parametrize crystal bases in representations and were later shown to define toric degenerations of Schubert varieties (\cite{Ca}) and Newton--Okounkov bodies of flag varieties (\cite{Ka}). This setting is extended by the birational sequences of~\cite{FaFL2} which provide toric degenerations, Newton--Okounkov bodies and monomial representation bases. A different approach is via semigroups of essential signatures and the associated polytopes which provide PBW-monomial bases, toric degenerations and Newton--Okounkov bodies (\cite{FFL3}). Yet another type of underlying datum is given by matching fields which again provide toric degenerations and Newton--Okounkov bodies (\cite{CM2,CM3}) as well as PBW-monomial bases (\cite{M3}). We now explain a certain discrepancy between general theory and explicit examples observed in this field which is one of the motivations for this project (see also \cite[page 1]{Che} for a much broader perspective).

In each of the above settings the respective datum must satisfy certain assumptions for the desired structures to exist. Finding instances in which these assumptions hold is highly nontrivial, let alone enumerating such instances. This leads to a shortage of known explicit examples of the mentioned structures, especially of concrete definitions which would work in some ``general'' situation rather than just ``small'' cases. Specifically, in the above works explicit examples which cover all type A flag varieties and/or all finite-dimensional $\fsl_n$-representations are essentially of two families. The first is the ``Gelfand--Tsetlin'' family of constructions arising as follows: the toric degeneration due to~\cite{Stu,GL,KM} is given by the Gelfand--Tsetlin polytope of~\cite{GT} which is also a Newton--Okounkov body (\cite{K}) and enumerates several combinatorial representation bases (\cite{GT,R,FaFL1}). 
The second is the ``FFLV'' (Feigin--Fourier--Littelmann--Vinberg) family: the PBW-monomial basis given by the FFLV polytope (\cite{FFL1}) and the corresponding Newton--Okounkov body and toric degeneration (\cite{FFL3}). 
In other classical types explicit examples are even scarcer with only some of the constructions in these two families having known counterparts.

In this regard, a notable advancement was made in~\cite{Fu} using marked chain-order polytopes (or MCOPs, introduced in \cite{FF}, see also~\cite{St,ABS}). Fujita shows that all MCOPs associated with the Gelfand--Tsetlin poset provide Newton--Okounkov bodies and toric degenerations of type $\rA_n$ flag varieties and also parametrize certain monomial bases. These polytopes form a large (i.e.\ growing with $n$) family that includes both the Gelfand--Tsetlin and FFLV polytopes. Thus, \cite{Fu} provides new large families of explicit examples of each of the three mentioned structures and generalizes the mentioned ``Gelfand--Tsetlin'' and ``FFLV''  families. The approach in~\cite{Fu} is via a certain geometrically defined valuation which provides the desired Newton--Okounkov body, the existence of the toric degeneration is then implied by a general result of~\cite{A}. The work~\cite{M3} was initiated with the aim of realizing these MCOP toric degenerations more directly as Gr\"obner (or sagbi) degenerations of the Pl\"ucker algebra, similarly to the classical approach in~\cite{GL}. This was done via a newly found connection between MCOPs and pipe dreams which was then also used to construct new PBW-monomial bases and standard monomial theories.

In this paper our goal is to address the discussed shortage of examples in Dynkin types B and C by constructing new families of toric degenerations, Newton--Okounkov bodies and PBW-monomial bases given by explicitly defined poset polytopes. There are substantial differences from type A presenting several challenges which will be outlined below. However, let us first mention that~\cite[Section 7]{Fu} explains in detail why its method does not generalize to type C. For us this serves as an additional motivation to look for an extension of the ``pipe dream method'' used in~\cite{M3} to type C and beyond.

Let us briefly overview our setup. We consider the type $\rC_n$ Gelfand--Tsetlin poset $P$, cf.\ \eqref{hasseC}. With a subset $O\subset P$ and an integral dominant weight $\la$ we associate the corresponding MCOP $\cQ_O(\la)\subset\bR^P$. The extremal choices of $O$ provide the type C Gelfand--Tsetlin polytope of~\cite{BZ} and the type C FFLV polytope of~\cite{FFL2}, other cases interpolate between these two. For a regular $\la$ the toric variety of $\cQ_O(\la)$ is cut out by $I_O$, a toric ideal in the polynomial ring $\bC[\cJ]$ with variables labeled by order ideals in $P$. The complete symplectic flag variety $F$ is cut out by the Pl\"ucker ideal $I$ that lies in a polynomial ring denoted by $S$. Furthermore, the negative roots are enumerated by poset elements:  $(i,j)\in P$ for every negative root vector $f_{i,j}\in\mathfrak{sp}_{2n}(\bC)$. Finally, we denote the irreducible representation by $V_\la$, its highest-weight vector by $v_\la$ and the corresponding equivariant line bundle on $F$ by $\cL_\la$. We prove the following.
\begin{customthm}{C}\hspace{-2mm}\footnotemark\label{thmC}
\footnotetext{Parts (a) and (c) of this theorem were announced without proof in the FPSAC extended abstract~\cite{FPSAC2024}.}
\hfill
\begin{enumerate}[label=(\alph*)]
\item (Theorem~\ref{degenmainC}.) A certain isomorphism $\psi:\bC[\cJ]\to S$ maps $I_O$ to an initial ideal of $I$. Hence, the toric variety of $\cQ_O(\la)$ with regular $\la$ is a flat degeneration of $F$.
\item (Theorem~\ref{mainbasisC}.) For a certain unimodular transformation $\xi$ of $\bR^P$ the vectors $\prod f_{i,j}^{x_{i,j}} (v_\la)$ with $x$ ranging over the lattice points of $\xi(\cQ_O(\la))$ form a basis in $V_\la$.
\item (Theorem~\ref{mainNOC}.) The polytope $\cQ_O(\la)$ is (up to translation) the Newton--Okounkov body of $F$ associated with the line bundle $\cL_\la$ and a certain valuation $\nu$.
\end{enumerate}    
\end{customthm}

The key ingredients of this theorem are the maps $\psi$, $\xi$  and $\nu$, which depend on $O$. These maps are defined in terms of diagrams which we call \textit{combinatorial type C pipe dreams}. These diagrams provide a method of associating a permutation of the set $\{1,\dots,n,-n,\dots,-1\}$ with every subset of $P$ (cf.\ Example~\ref{pipedreamex}). 
We use the word ``combinatorial'' to avoid confusion. In this paper, pipe dreams arise as a combinatorial tool which allows one to pass from poset polytopes to permutations and, eventually, to Lie-theoretic structures. While the diagrammatic definition of combinatorial type C pipe dreams is very similar to that of standard pipe dreams (originating in~\cite{BB}), their algebraic applications here are rather different. The reader should not expect applications to Schubert calculus similar to the results in \cite{FK1,KiNa,SmTu,FuNi,Fujita2022}: works that consider various other type C analogs of pipe dreams. 

A crucial difference from type A is that the Pl\"ucker algebra $S/I$ does not have a similarly convenient realization as a subalgebra of a polynomial ring generated by determinants, which would allow for explicit sagbi degenerations. It can be realized as a quotient of such a subalgebra but constructing explicit degenerations of this quotient is difficult. Instead, we use a certain Schubert variety as an intermediate step between the flag variety and the toric varieties. Specifically, we notice that the type $\rC_n$ flag variety can be degenerated into a type $\rA_{2n-1}$ Schubert variety (see Corollary~\ref{schubert}, this fact is somewhat reminiscent of the results in~\cite{CL} that realize PBW-degenerate flag varieties of types A and C as Schubert varieties). Now, this Schubert variety is realized by an algebra $\cR$ that can be generated by determinants, this allows us to obtain the toric rings $S/I_O$ as further sagbi degenerations of $\cR$. The latter sagbi degenerations then play a major role in the proofs of all parts of Theorem~\ref{thmC}. In particular, we show that general sagbi degenerations of $\cR$ provide PBW-monomial bases (Lemma~\ref{sagbiToMonomials}).
We also show that for every $O$ we obtain a standard monomial theory in the Pl\"ucker algebra which can be interpreted in terms of a certain class of Young tableaux (Corollary~\ref{standmon}).

In type B the first problem is to define analogs of MCOPs which would form a family interpolating between the type B Gelfand--Tsetlin polytope of~\cite{BZ} and the type B FFLV polytope defined in~\cite{MtypeB}. This family is parametrized by subsets $O$ of the same poset $P$, for a weight $\la$ we denote the corresponding polytope $\cQ^B_O(\la)$. The definition is similar to that of MCOPs but involves scaling along certain coordinates. Although $\cQ^B_O(\la)$ is not necessarily a lattice polytope, its intersection with a certain lattice contains $\dim V_\la$ points. This allows for the desired Lie-theoretic applications. 

Now, in type B we find the Pl\"ucker algebra to be unfeasibly complicated for our purposes and instead we choose an approach via essential signatures similar to~\cite{FFL3}. We show that the lattice points of a transformed version of $\cQ^B_O(\la)$ form the set of essential signatures: perhaps surprisingly, the most difficult part here is dealing with the case of fundamental $\la$. The geometric results concerning toric degenerations and Newton--Okounkov bodies are then derived using the embedding of the partial flag variety $F_\la$ into the $\bP(V_\la)$. This also leads to a family of standard tableau theories (Remark~\ref{tableauxB}). The main results are summed up by
\begin{customthm}{B}\label{thmB}
\hfill
\begin{enumerate}[label=(\alph*)]
\item (Theorem~\ref{essentialB}.) For a certain unimodular transform $\Pi^\rB_O(\la)$ of $\cQ^B_O(\la)$ the vectors $\prod f_{i,j}^{x_{i,j}} (v_\la)$ with $x$ ranging over the lattice points in $\Pi^\rB_O(\la)$ form a basis in $V_\la$.
\item (Theorem~\ref{degenmainB}.) The toric variety of $\cQ^\rB_O(\la)$ is a flat degeneration of $F_\la$.
\item (Theorem~\ref{mainNOB}.) For a regular $SO_{2n+1}$-weight $\la$ polytope of $\cQ^\rB_O(\la)$ is (up to translation) a Newton--Okounkov body of the complete flag variety $F_\la$.
\end{enumerate}    
\end{customthm}

The main combinatorial ingredients of all three constructions in Theorem~\ref{thmB} are defined using the same notion of pipe dreams as in type C.

We conclude the introduction with several remarks. First, we comment on the generality of specific parts of Theorems~\ref{thmB} and~\ref{thmC}. A common approach in the study of toric degenerations and Newton--Okounkov bodies of flag varieties is to concentrate on complete flags even if the results have straightforward extensions to the partial case, cf.\ \cite{Ca,KM,Ka,Ki17,Fu} and others. This is somewhat similar to our approach: we prioritize complete flag varieties, however, Theorem~\ref{thmC}(c) and Theorem~\ref{thmB}(b) are proved for arbitrary partial flags since this is easily achieved within the same framework. In contrast, proving Theorem~\ref{thmC}(a) or Theorem~\ref{thmB}(c) in that generality, while presumably possible, would require a considerable amount of additional technical discussion and is avoided in the interests of readability.


Next, for general $O$ the structures provided by Theorems~\ref{thmC} and~\ref{thmB} are new but some of those obtained in the special cases of Gelfand--Tsetlin and FFLV polytopes appear in the literature. In Type C the resulting six structures (or variations thereof) are due to~\cite{Ca,FFL2,Ka,FFL3,Ki19,MY}. In type B the basis for the Gelfand--Tsetlin case and the toric degeneration and Newton--Okounkov body for the FFLV case appear to be new, the three remaining cases are due to~\cite{Ca,Ka,MtypeB}.

Finally, the first obstacle to extending these results to type D is finding an appropriate family of polytopes, in particular, a type D analog for FFLV polytopes is not known.

\section{Type C}

\subsection{Type C Lie algebras and representations}

Choose an integer $n\ge 1$ and consider the Lie algebra $\fg=\mathfrak{sp}_{2n}(\bC)$ with Cartan subalgebra $\fh$. Denote the simple roots by $\alpha_1,\dots,\alpha_n\in\fh^*$. We have an (orthogonal) basis $\varepsilon_1,\dots,\varepsilon_n$ in $\fh^*$ such that $\alpha_i=\varepsilon_i-\varepsilon_{i+1}$ for $i\le n-1$ and $\alpha_n=2\varepsilon_n$. The positive roots are indexed by pairs of integers $i,j$ such that $i\in[1,n]$ and $j\in[i+1,n]\cup[-n,-i]$. The root $\alpha_{i,j}$ is equal to $\varepsilon_i-\varepsilon_j$ when $j>0$ and to $\varepsilon_i+\varepsilon_j$ when $j<0$. In particular, $\alpha_i=\alpha_{i,i+1}$ for $i\le n-1$, $\alpha_n=\alpha_{n,-n}$ and the roots $\alpha_{i,-i}$ are long. 

Denote the fundamental weights by $\om_1,\dots,\om_n$. We write $(a_1,\dots,a_n)$ to denote the weight $\la=a_1\om_1+\dots+a_n\om_n$. For an integral dominant $\la$ (i.e.\ all $a_i\in\bZ_{\ge0}$) the irreducible representation with highest weight $\la$ is denoted by $V_\la$, its highest-weight vector by $v_\la$.

In terms of the orthogonal basis one has $\om_i=\varepsilon_1+\dots+\varepsilon_i$. For a weight $\la=(a_1,\dots,a_n)$ we denote its coordinates with respect to the basis $\varepsilon_1,\dots,\varepsilon_n$ by $(\la(1),\dots,\la(n))$ so that $\la(i)=a_i+\dots+a_n$.

\subsection{Type C poset polytopes}

Consider the totally ordered set \[(N,\lessdot)=\{1\lessdot\dots\lessdot n\lessdot -n\lessdot\dots\lessdot -1\}.\] 

\begin{definition}
The \textit{type C Gelfand--Tsetlin poset} $(P,\prec)$ consist of pairs of integers $(i,j)$ such that $i\in[1,n]$ and $i\le |j|\le n$. The order relation is given by: $(i_1,j_1)\preceq(i_2,j_2)$ if and only if $i_1\le i_2$ and $j_1\ledot j_2$.
\end{definition}

Below is the Hasse diagram of $(P,\prec)$ for $n=2$.
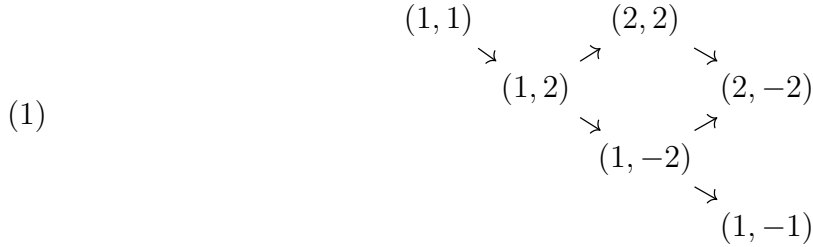
\begin{equation}
\begin{tikzcd}[row sep=tiny,column sep=0]\label{hasseC}
(1,1)\arrow[rd]&&(2,2)\arrow[rd]\\
&(1,2)\arrow[rd]\arrow[ru]&&(2,-2)\\
&&(1,-2)\arrow[ru]\arrow[rd]\\
&&&(1,-1)
\end{tikzcd}
\end{equation}

Marked chain-order polytopes (MCOPs) were introduced in the paper~\cite{FF} and the polytopes we consider below are a special case of this notion. 

Let $A\subset P$ be the set of all $(i,i)$. Throughout the section we fix a subset $O\subset P$ containing $A$, the polytopes are determined by this choice. We give two equivalent definitions, a proof of the equivalence can be found in~\cite[Subsection 3.5]{FM2}. It should be noted that in the terminology of~\cite{FM2} the polytope defined below is the marked chain-order polytope defined by the following data. One considers the poset $(P',\prec)$ obtained from $(P,\prec)$ by adding an element $p$ satisfying $p\succ (i,j)$ for all $(i,j)$, the subset of marked elements $A\cup\{p\}$, the marking associating $\la(i)$ to $(i,i)$ and 0 to $p$ and the partition $P\bs A=P\bs O\sqcup O\bs A$.

\begin{definition}\label{mcopHdefC}
For an integral dominant $\fg$-weight $\la$ the MCOP (or \textit{type C poset polytope}) $\cQ_O(\la)\subset\bR^P$ consists of points $x$ such that:
\begin{itemize}
\item $x_{i,i}=\la(i)$ for all $1\le i\le n$,
\item all $x_{i,j}\ge 0$,
\item for every chain $(p,q)\prec(i_1,j_1)\prec\dots\prec(i_m,j_m)\prec(r,s)$ with $(p,q)\in O$, $(r,s)\in P$ and all $(i_l,j_l)\notin O$ one has \[x_{i_1,j_1}+\dots+x_{i_m,j_m}\le x_{p,q}-x_{r,s},\]
\end{itemize}    
\end{definition}

If $O=P$, then $\cQ_O(\la)$ is a Berenstein--Zelevisnky polytope constructed in~\cite{BZ}, it is also known as the \textit{type C Gelfand--Tsetlin polytope}. It consists of points $x$ with $x_{i,i}=\la(i)$ and $x_{i,j}\ge x_{i',j'}\ge 0$ whenever $(i,j)\preceq(i',j')$. If $O=A$, then $\cQ_O(\la)$ is the \textit{type C Feigin--Fourier--Littelmann--Vinberg (FFLV) polytope} of \cite{FFL2} given by restricting the sum over every chain in $P\bs A$. Other cases can be said to interpolate between these two.

Let $\cJ$ denote the set of order ideals (lower sets) in $(P,\prec)$. For $k\in[1,n]$ let $\cJ_k$ consist of $J\in\cJ$ such that $|J\cap A|=k$, i.e.\ $J$ contains $(k,k)$ but not $(k+1,k+1)$. For $X\subset P$ let $\one_X\in\bR^P$ denote its indicator vector.
\begin{definition}\label{mcopdefC}
For $J\in\cJ$ denote \[M_O(J)=(J\cap O)\cup\max\nolimits_{\prec}(J)\] where $\max_\prec$ is the subset of $\prec$-maximal elements. The type C poset polytope $\cQ_O(\om_k)$ is the convex hull of all $\one_{M_O(J)}$ with $J\in\cJ_k$. For $\la=(a_1,\dots,a_n)$ the polytope $\cQ_O(\la)$ is the Minkowski sum \[a_1\cQ_O(\om_1)+\dots+a_n\cQ_O(\om_n).\]
\end{definition}

Since $\cQ_O(\om_k)$ is a 0/1-polytope, the set $\{1_{M_O(J)}\}_{J\in\cJ_k}$ is its vertex set and also its set of lattice points. It is not hard to check that $|\cJ_k|={2n\choose k}-{2n\choose k-2}=\dim V_{\om_k}$. More generally, the following holds.

\begin{lemma}[{\cite[Corollary 3.1.9]{FM2}}]\label{pointdecomp}
For an integral dominant $\la=(a_1,\dots,a_n)$ and a lattice point $x\in\cQ_O(\la)$ there exists a unique sequence of order ideals $J_1\subset\cdots\subset J_m$ in $\cJ$ such that \[x=\one_{M_O(J_1)}+\dots+\one_{M_O(J_m)}\] and for every $k$ exactly $a_k$ of the $J_i$ lie in $\cJ_k$.
\end{lemma}

This shows that the number of lattice points in $\cQ_O(\la)$ does not depend on $O$. A fundamental property of the type C Gelfand--Tsetlin polytopes $\cQ_P(\la)$ and type C FFLV polytopes $\cQ_A(\la)$ is that both have $\dim V_\la$ lattice points. We deduce
\begin{cor}
$\cQ_O(\la)$ contains exactly $\dim V_\la$ lattice points.   
\end{cor}

Another rather helpful feature of the lattice point sets is the so-called Minkowski sum property (which also follows from Lemma~\ref{pointdecomp}):
\begin{theorem}[{\cite[Theorem 2.8]{FFP}}]\label{minkowskiC}
For any integral dominant weights $\la$ and $\mu$ one has \[\cQ_O(\la+\mu)\cap\bZ^P=\cQ_O(\la)\cap\bZ^P+\cQ_O(\mu)\cap\bZ^P.\]
\end{theorem}


\subsection{Combinatorial type C pipe dreams}

Let $\mathcal S_N$ denote the group of all permutations (automorphisms) of the set $N$. For $(i,j)\in P$ let $s_{i,j}\in \mathcal S_N$ denote the transposition which exchanges $i$ and $j$ and fixes all other elements ($s_{i,i}=\id$). 

\begin{definition}
For $M\subset P$, let $w_M\in \mathcal S_N$ be the product of all $s_{i,j}$ with $(i,j)\in M$ ordered first by $i$ increasing from left to right and then by $j$ increasing with respect to $\lessdot$ from left to right.
\end{definition}

Note that $w_M$ is fully determined by $M\bs A$. However, it will be more convenient for us to associate permutations to subsets of $P$ rather than $P\bs A$.

A diagrammatic approach to this notion which motivates the name ``pipe dreams'' (this term is due to~\cite{KnM}) is as follows. Let the poset $(P,\prec)$ be visualized as in~\eqref{hasseC}. The \textit{combinatorial type C pipe dream} (or just \textit{pipe dream} for brevity) of the set $M\subset P$ consists of $2n$ \textit{pipes} enumerated by $N$. Each pipe can be thought of as polygonal curve with vertices at elements of $P$. For $i\in[1,n]$ the $i$th pipe enters the element $(i,-i)$ from the \textbf{bottom}-right and continues in this direction until it reaches an element of $M\cup A$, after which it turns left and continues going to the bottom-left until it reaches an element of $M$, after which it turns right and again continues to the top-left until it reaches an element of $M\cup A$, etc. Meanwhile, for $i\in[-n,-1]$ the $i$th pipe enters the element $(-i,i)$ from the \textbf{top}-right and then also turns each time it encounters an element of $M\cup A$. It is then easy to see that the last element of $P$ passed by the pipe will be $(1,w_M(i))$. 

\begin{example}\label{pipedreamex}
The pipe dream of the set $M=\{(1,1),(1,3),(1,-2),(2,2),(2,3),(3,-3)\}$ is shown below with each pipe in its own colour (here $n=3$). One obtains \[w_M(1,2,3,-3,-2,-1)=(-2,1,-3,2,3,-1)\] which agrees with $w_M=s_{1,1}s_{1,3}s_{1,-2}s_{2,2}s_{2,3}s_{3,-3}$.
\begin{center}
\begin{tikzcd}[row sep=tiny,column sep=0]
&(1,1)\ar[green]{ld}\ar[green]{ld}&&(2,2)\ar[magenta]{ld}&&\color{lightgray}{(3,3)}\ar[magenta]{ld}&&\phantom{(1,1)}\ar[magenta]{ld}\\
\phantom{(1,1)}&&\color{lightgray}(1,2)\ar[green]{lu}\ar[magenta]{ld}&&(2,3)\ar[green]{ld}\ar[magenta]{lu}&&{(3,-3)}\ar[blue]{ld}\ar[magenta]{lu}&&\phantom{(1,1)}\\
&\phantom{(1,1)}&&(1,3)\ar[green]{lu}\ar[orange]{ld}&&\color{lightgray}{(2,-3)}\ar[green]{lu}\ar[blue]{ld}&&\phantom{(1,1)}\ar[blue]{lu}\ar[orange]{ld}\\
&&\phantom{(1,1)}&&\color{lightgray}(1,-3)\ar[blue]{ld}\ar[orange]{lu}&&\color{lightgray}(2,-2)\ar[green]{lu}\ar[orange]{ld}&&\phantom{(1,1)}\\
&&&\phantom{(1,1)}&&(1,-2)\ar[red]{ld}\ar[orange]{lu}&&\phantom{(1,1)}\ar[green]{lu}\ar[cyan]{ld}\\
&&&&\phantom{(1,1)}&&\color{lightgray}(1,-1)\ar[red]{lu}\ar[cyan]{ld}\\
&&&&&\phantom{(1,1)}&&\phantom{(1,1)}\ar[red]{lu}
\end{tikzcd}
\end{center}
\end{example}

\begin{example}
    When $n=2$, a total of 12 different permutations arise as $w_M$ for the 16 subsets $M\subset P\backslash A$.
    If $w_M$ is applied to $(1,2,-2,-1)$ elementwise, each of
    \[(-2,2,1,-1), (-2,1,2,-1), (-1,2,1,-2), (-1,1,2,-2)\]
    is obtained for two distinct $M$, while 8 more are given by a single $M$:
    \begin{multline*}
    (1,2,-2,-1),(1,-2,2,-1),(2,1,-2,-1),(2,-2,1,-1),\\(-1,2,-2,1),(-1,-2,2,1),(-1,-2,1,2),(-1,1,-2,2).
    \end{multline*}
\end{example}


Next, more formally, for $i\in N$ and $M\subset P$ we will view the $i$th pipe of $M$ as a sequence $(i_1,j_1)\succ\dots\succ(i_m,j_m)$ in $P$. This sequence forms a saturated chain, starts with $(i_1,j_1)=(|i|,-|i|)$ and and ends with $(i_m,j_m)=(1,w_M(i))$.

\begin{remark}
Pipe dreams of type C can be thought of as special cases of classical type A pipe dreams. Indeed, consider the type $\rA_{2n-1}$ Gelfand--Tsetlin poset $Q$ consisting of $(i,j)$ with $1\le i\le j\le 2n$. One can view $P$ as the ``left half'' of $Q$ by identifying $(i,j)\in P$ with $(i,2n+1+j)\in Q$ when $j<0$. The $2n$ pipes in the type C pipe dream of $M\subset P$ will just be end parts of the $2n$ pipes in the type A pipe dream of $M\subset Q$.
\end{remark}

The choice of the subset $O$ provides a ``twisted'' correspondence between subsets and permutations and also associates an element of $N$ with every element of $P$. These correspondences play a key role in our construction.
\begin{definition}\label{twisted}
\hfill
\begin{itemize}
\item For $M\subset P$ denote $w^O_M=w_O^{-1}w_M$. 
\item For $J\in\cJ$ we use the shorthand $w^{O,J}=w^O_{M_O(J)}$. 
\item For $(i,j)\in P$ let $\langle i,j\rangle$ denote the order ideal of all $(i',j')\preceq(i,j)$. 
\item We set $r(i,j)=w^{O,\langle i,j\rangle}(i)$.
\end{itemize}
\end{definition}

In other words, $w^O_M(i)$ is the number of the pipe in the pipe dream of $O$ which ends in the same element as the $i$th pipe of $M$. 
To understand the value $r(i,j)$ one may think of a ``pipe'' which, instead of starting in some $(l,-l)$, starts from $(i,j)$ going to the bottom-left and then turns at elements of $O$. The value $r(i,j)$ is the number of the pipe in the pipe dream of $O$ which ends in the same element as this ``pipe''.

\begin{example}\label{rex}
Consider $n=3$ and \[O=\{(1,1), (1,3), (1,-2), (2,2), (2,3), (3,3), (3,-3)\}.\] The pipe dream of $O$ coincides with the pipe dream of the set $M$ in Example~\ref{pipedreamex}, hence $w_O(1,2,3,-3,-2,-1)=(-2,1,-3,2,3,-1)$. One can compute \[w_{M_O(\langle 2,-3 \rangle)}(1,2,3,-3,-2,-1)=(3,-3,2,1,-2,-1),\] \[w^{O, \langle 2,-3 \rangle}(1,2,3,-3,-2,-1)=(-2,3,-3,2,1,-1).\] In particular, $r(2,-3)=3$. Furthermore, one may check that $r(1,1)=2$, $r(1,2)=-3$, $r(1,3)=-2$, $r(1,-3)=3$, $r(1,-2)=1$, $r(1,-1)=-1$, $r(2,2)=-3$, $r(2,3)=2$, $r(2,-3)=3$, $r(2,-2)=-2$, $r(3,3)=-3$, $r(3,-3)=3$ (cf.\ Proposition~\ref{wrproperties}(b)).
\end{example}

We list some properties of these notions which we will use, occasionally without reference. On a general note, let us mention that while we aim to keep our combinatorial arguments concerning pipe dreams rigorous and complete, it should still be helpful to visualize these arguments and the reader is encouraged to do so.
\begin{proposition}\label{wrproperties}
\hfill
\begin{enumerate}[label=(\alph*)]
\item\label{wOisr} For $J\subset \cJ_k$ and $i\in[1,k]$ consider the $\lessdot$-maximal $j$ with $(i,j)\in M_O(J)$. Then $w^{O,J}(i)=r(i,j)$. 
\item For $i\in[1,n]$ the set $\{r(i,j)\}_{j\in [i,n]\cup[-n,-i]}$ coincides with $[i,n]\cup[-n,-i]$.
\item\label{rijgei} For any $J\in\cJ_k$ and $i\in[1,k]$ one has $|w^{O,J}(i)|\ge i$.
\item For $l\in[1,n]$ the $l$th pipe of $O$ contains all $(i,j)$ with $r(i,j)=l$.
\item\label{rbound1} For $(i,j)\in P$ one has $|r(i,j)|\ge i$ and $|r(i,j)|\ledot -j$.
\item\label{rbound2} For $(i,j)\in O$ suppose there exists a $\lessdot$-minimal $j'\gtrdot j$ with $(i,j')\in O$. Then $|r(i,j)|\ledot -j'$.
\end{enumerate}
\end{proposition}
\begin{proof}
\hfill

(a) It is evident that the $i$th pipe of $M_O(J)$ coincides with the $i$th pipe of $M_O(\al i,j\ar)$, the claim follows by Definition~\ref{twisted}.

(b) First, note that the values $r(i,j)$ for a given $i$ must be pairwise distinct. Indeed, suppose the $i$th pipes of $M_O(\al i,j_1\ar)$ and $M_O(\al i,j_2\ar)$ end in the same element. Let $(i',j')$ be the $\prec$-maximal element contained in both pipes. For either pipe the element preceding $(i',j')$ is determined by the direction in which the pipe leaves $(i',j')$ and whether $(i',j')$ lies in the respective set. However, we have $i'<i$ and $j'\ledot \min_\lessdot(j_1,j_2)$ which means that $(i,j')$ lies in $M_O(\al i,j_1\ar)$ if and only if it lies in $M_O(\al i,j_2\ar)$. The direction in which the pipe leaves $(i',j')$ must be the same in both cases by our choice of $(i',j')$. Hence, we obtain a contradiction unless the pipes fully coincide, i.e.\ $j_1=j_2$.

Now, the $r(i,j)$th pipe of $O$ ends in same element as the $i$th pipe of $M_O(\al i,j\ar)$. Consider the $\prec$-maximal element $(i',j')$ contained in both of these pipes. Similarly to the above we can't have $(i',j')\prec (i,j)$ because such a $(i',j')$ lies in $M_O(\al i,j\ar)$ if and only if it lies in $O$. Hence $(i,j)$ lies in both pipes. This implies $(i,j)\preceq(|r(i,j)|,-|r(i,j)|)$ and $|r(i,j)|\ge i$.

(c) This is immediate from the previous two parts.

(d) In the proof of part (b) we have seen that the $r(i,j)$th pipe of $O$ contains $(i,j)$.

(e) In the proof of part (b) we already showed that $(i,j)\preceq(|r(i,j)|,-|r(i,j)|)$ which implies both inequalities.

(f) In this case the $r(i,j)$th pipe of $O$ contains $(i,j')$: this is the first element in the pipe of the form $(i,l)$. Consequently $(i,j')\preceq(|r(i,j)|,-|r(i,j)|)$.
\end{proof}

\subsection{Transformed poset polytopes}

Pipe dreams allow us to define a unimodular transformation of the MCOP which is sometimes more convenient to work with.

\begin{definition}\label{defxi}
For $(i,j)\in P$ let $\epsilon_{i,j}$ denote the basis vector in $\bR^P$ corresponding to $(i,j)$. Let $\xi:\bR^P\to \bR^P$ be the linear map defined on $\epsilon_{i,j}$ as follows. If $j\neq i$, consider the $\lessdot$-maximal $j'\lessdot j$ such that $(i,j')\in O$ and set
\[
\xi(\epsilon_{i,j})=
\begin{cases}
\epsilon_{i,r(i,j)}-\epsilon_{i,r(i,j')}&\text{ if }i\neq j,\\
\epsilon_{i,r(i,i)}&\text{ if }i=j.
\end{cases}
\]
\end{definition}

Obviously, one may order the $\epsilon_{i,j}$ so that $\xi(\epsilon_{i,j})$ is equal to $\epsilon_{i,r(i,j)}$ plus a linear combination of (zero or one) subsequent vectors. This shows that the matrix of $\xi$ is the product of a unitriangular matrix and a permutation matrix, hence $\xi$ is unimodular. 
We denote $\Pi_O(\la)=\xi(\cQ_O(\la))$, this is a unimodular transform of the MCOP. The following proposition describes the lattice points of $\Pi_O(\om_k)$. 

\begin{proposition}\label{pointimage}
For $J\in\cJ_k$ the coordinate $\xi(\one_{M_O(J)})_{i,j}$ is 1 if $j=w^{O,J}(i)$ and $i\in[1,k]$ and 0 otherwise: \[\xi(\one_{M_O(J)})=\one_{\left\{\left(1,w^{O,J}(1)\right),\dots,\left(k,w^{O,J}(k)\right)\right\}}.\]
\end{proposition} 
\begin{proof}
By the definitions of $M_O(J)$ and $\xi$ for any $i\in[1,k]$ one has \[\xi\left(\sum_{(i,j)\in M_O(J)}\epsilon_{i,j}\right)=\epsilon_{i,r(i,j')}\] where $j'$ is $\lessdot$-maximal among $j$ with $(i,j)\in M_O(J)$. However, by Proposition~\ref{wrproperties}\ref{wOisr} we have $r(i,j')=w^{O,J}(i)$ and the claim follows.
\end{proof}

\begin{example}\label{xiex}
Consider $n=3$ and $O=\{(1,1), (1,3), (1,-2), (2,2), (2,3), (3,3), (3,-3)\}$ as in Example~\ref{rex} and $J=\langle (2,-3)\rangle\in \cJ_2$. Then 
$\xi(\one_{M_O(J)})=\one_{\{(1,-2),(2,3)\}}$.
\end{example}

\begin{example}\label{xiexfflv}
Note that for $O=A$ one has $w_O=\id$ and $r(i,j)=j$ for any $(i,j)$. Consequently, in this case $\xi(x)_{i,j}=x_{i,j}$ for any $(i,j)\notin A$ while $\xi(x)_{i,i}$ is found from $\sum_j \xi(x)_{i,j}=x_{i,i}=\la(i)$ for $x\in\cQ_O(\la)$. This means that the polytopes $\cQ_A(\la)$ and $\Pi_A(\la)$ are almost the same: they project into the same polytope in $\bZ^{P\bs A}$ and this projection is unimodularly equivalent to both. For $O=P$ the map $\xi$ is less trivial and the polytope $\Pi_P(\la)$ differs substantially from $\cQ_P(\la)$, see Examples~\ref{psiexgt} and~\ref{basisexgt}.
\end{example}

Note that, in view of Definition~\ref{mcopdefC}, all $\cQ_O(\la)$ with regular $\la$ have the same normal fan. Hence, the same holds for the $\Pi_O(\la)$. We give a multiprojective realization of the toric variety of $\Pi_O(\la)$ with regular $\la$ (which is isomorphic to the toric variety of $\cQ_O(\la)$). 

Consider the product \[\bP_\cJ=\bP(\bC^{\cJ_1})\times\dots\times\bP(\bC^{\cJ_n}).\] Its multihomogeneous coordinate ring is the polynomial ring $\bC[\cJ]$ in variables $X_J$ where $J\in\cJ$ is nonempty. Consider also the polynomial ring $\bC[P]$ in variables $z_{i,j}$ with $(i,j)\in P$. Let $\varphi_O:\bC[\cJ]\to\bC[P]$ be given by \[\varphi_O(X_J)=\prod_{i=1}^{|J\cap A|} z_{i,w^{O,J}(i)}.\]
\begin{theorem}
For regular $\la$ the toric variety of $\Pi_O(\la)$ is isomorphic to the zero set of the ideal $I_O=\ker\varphi_O$ in $\bP_\cJ$.
\end{theorem}
\begin{proof}
For $k\in[1,n]$ consider the subring $\bC[\cJ_k]\subset\bC[\cJ]$ generated by $X_J$ with $J\in\cJ_k$. Then the kernel of the restriction $\varphi_k$ of $\varphi_O$ to $\bC[\cJ_k]$ cuts out the toric variety of $\Pi_O(\om_k)$ in $\bP(\bC^{\cJ_k})$. Indeed, by Proposition~\ref{pointimage} the lattice points of $\Pi_O(\om_k)$ are enumerated by $\cJ_k$ and $\varphi_k$ maps $X_J$ to the exponential of the corresponding point. Furthermore, in view of Proposition~\ref{minkowskiC} the polytope $\Pi_O(\om_k)$ is normal, hence the kernel of such a map defines its toric variety.

Now, the toric variety of the Minkowski sum \[\Pi_O(\om_1)+\dots+\Pi_O(\om_n)=\Pi_O((1,\dots,1))\] has a standard multiprojective realization. It is given by the kernel of the map $\varphi_O$ because the latter is obtained by combining the maps $\varphi_k$ corresponding to the summands, see, for instance,~\cite[Lemma 1.8.3]{FM2} for a general statement and proof. 
\end{proof}

\subsection{Gr\"obner and sagbi degenerations}

A \textit{monomial order} on a polynomial ring $\bC[x_s]_{s\in S}$ is a partial order $<$ on the set of monomials with the following two properties.
\begin{itemize}
\item The order is multiplicative: for monomials $M_1,M_2$ and $s\in S$ one has $M_1<M_2$ if and only if $M_1x_s<M_2x_s$. 
\item The order is weak, i.e.\ incomparability is an equivalence relation.
\end{itemize}
Note that every total order is weak and weak orders are precisely the  pullbacks of total orders. Moreover, every monomial order can be obtained by applying a monomial specialization and then comparing the results lexicographically. We will not be using this general fact, see~\cite[Theorem 1.2]{KNN} for a proof and further context (there the term ``monomial preorder'' is used instead). 

For a monomial order $<$ and a polynomial $p\in \bC[x_s]_{s\in S}$ the initial part $\initial_< p$ is equal to the sum of those monomials occurring in $p$ which are maximal with respect to $<$ taken with the same coefficients as in $p$. For any subspace $U\subset\bC[x_s]_{s\in S}$ its initial subspace $\initial_< U$ is the linear span of $\{\initial_< p\}_{p\in U}$. One easily checks that the initial subspace of an ideal is an ideal (the initial ideal) and the initial subspace of a subalgebra is a subalgebra (the initial subalgebra). 

\begin{definition}
For a monomial order $<$ on $\bC[x_s]_{s\in S}$ and a subalgebra $U\subset\bC[x_s]_{s\in S}$ a generating set $\{p_t\}_{t\in T}\subset U$ is called a \textit{sagbi basis} of $U$ if $\{\initial_< p_t\}_{t\in T}$ generates $\initial_< U$.
\end{definition}

Next, consider another polynomial ring $R=\bC[y_t]_{t\in T}$  and a homomorphism $\varphi:R\to \bC[x_s]_{s\in S}$. Let $<$ be a \textbf{total} monomial order on $\bC[x_s]_{s\in S}$ and let $\varphi_<:R\to\bC[x_s]_{s\in S}$ be the homomorphism mapping $y_t$ to $\initial_<\varphi(y_t)$. Consider the pullback of $<$ with respect to $\varphi_<$: set $M_1<^\varphi M_2$ if and only if $\varphi_<(M_1)<\varphi_<(M_2)$. Evidently, $<^\varphi$ is a monomial order. A standard fact relates initial ideals and initial subalgebras (cf.\ e.g.\ \cite[Lemma 1.5.3]{BCCV}):
\begin{proposition}\label{idealsubalg}
If the ideal $\ker\varphi$ is homogeneous and the set $\{\varphi(y_t)\}_{t\in T}$ is a sagbi basis of $\varphi(R)$ with respect to $<$, then $\ker\varphi_<=\initial_{<^\varphi}\ker\varphi$.
\end{proposition}
\begin{proof}
For $p\in R$ note that if $\varphi_<(\initial_{<^\varphi}p)\neq 0$, then it is a scalar multiple of a monomial and, moreover, $\varphi_<(\initial_{<^\varphi}p)=\initial_< \varphi(p)$. We deduce that $\varphi_<(\initial_{<^\varphi}p)=0$ if $\varphi(p)=0$, i.e.\ $\initial_{<^\varphi}\ker\varphi\subset\ker\varphi_<$. 

In view of the sagbi basis assumption we have $\varphi_<(R)=\initial_<\varphi(R)$. Since $\ker\varphi$ is homogeneous, the algebras $\varphi_<(R)$ and $R/\initial_{<^\varphi}\ker\varphi$ are graded with finite-dimensional components, also the latter surjects onto the former. However, passing to initial subspaces preserves graded dimensions. Thus, the surjection is an isomorphism since \[\grdim\varphi_<(R)=\mathrm{grdim}\initial_<\varphi(R)=\mathrm{grdim}\varphi(R)=\mathrm{grdim}(R/\initial_{<^\varphi}\ker\varphi).\qedhere\]
\end{proof}

The geometric motivation for considering initial ideals and subalgebras is that they provide flat degenerations. The following theorem is essentially classical, for a proof in the setting of partial monomial orders see~\cite[Theorem 3.2, Proposition 3.4]{KNN}.

\begin{theorem}\label{flatfamily}
For every monomial order $<$ on R and ideal $I\subset R$ there exists a flat $\bC[t]$-algebra $\mathcal A$ such that $\mathcal A/\langle t\rangle\simeq R/\initial_< I$ while for any nonzero $c\in\bC$ one has $\mathcal A/\langle t-c\rangle\simeq R/I$. 
\end{theorem}

Suppose the ideal $I$ is homogeneous. In geometric terms the above theorem means that we have a flat family over $\mathbb A^1$ for which the fiber over 0 is isomorphic to $\Proj R/\initial_< I$ while all other fibers are isomorphic to $\Proj R/I$. This flat family is known as a \textit{Gr\"obner degeneration} of the latter scheme into the former. In the setting of Proposition~\ref{idealsubalg} we obtain a flat family with fiber over 0 isomorphic to $\Proj(\initial_< \varphi(R))$ and other fibers isomorphic to $\Proj\varphi(R)$, this special case is known as a \textit{sagbi degeneration}. Note that since the order $<$ is total, $\initial_< \varphi(R)$ is generated by a finite set of monomials, i.e.\ it is a toric ring and $\ker\varphi_<$ is a toric ideal. This means that the fiber over 0 is a toric variety and we have a toric degeneration. In these constructions $\Proj$ can be replaced with $\MultiProj$ if $I$ is multihomogeneous with respect to some grading.

\subsection{The type C Pl\"ucker algebra and the Schubert degeneration}

Consider the polynomial ring $S=\bC[X_{i_1,\dots,i_k}]_{k\in[1,n],\{i_1\lessdot\dots\lessdot i_k\}\subset N}$. This ring is the multihomogeneous coordinate ring of the product \[\bP_\rA=\bP(\bC^{2n\choose 1})\times\dots\times\bP(\bC^{2n\choose n}).\] For every variable $X_{i_1,\dots,i_k}$ in $S$ we consider the $k\times k$ determinant \[C_{i_1,\dots,i_k}=|z_{i,j}|_{i=1,\dots,k,j=i_1,\dots,i_k}.\] Let $I_\rA\subset S$ denote the kernel of the homomorphism $\varphi_\rA$ from $S$ to $\bC[z_{i,j}]_{i\in[1,n],j\in N}$ taking $X_{i_1,\dots,i_k}$ to $C_{i_1,\dots,i_k}$. As is well known, the zero set of $I_\rA$ in $\bP_\rA$ is the partial flag variety of $GL(\bC^N)$ of signature $(1,\dots,n)$, hence the subscript. 

Next, for $\{i_1\lessdot\dots\lessdot i_k\}\subset N$ with $k\in[0,n-2]$ consider the linear expression \[L_{i_1,\dots, i_k}=X_{i_1,\dots,i_k,1,-1}+\dots+X_{i_1,\dots,i_k,n,-n}\in S.\] Here we use the standard convention $X_{i_1,\dots,i_k}=(-1)^\sigma X_{i_{\sigma(1)},\dots,i_{\sigma(k)}}$ for a permutation $\sigma\in \mathcal S_k$ (in particular, $X_{i_1,\dots,i_k}=0$ if two subscripts coincide). Let $\mathfrak L$ denote the linear span of all $L_{i_1,\dots, i_k}$. Let $F$ denote the complete flag variety $G/B$ of the group $G=Sp_{2n}(\bC)$ with $B\subset G$ the Borel subgroup.
\begin{theorem}[see, e.g., \cite{DC}]
The zero set of $I_\rA+\langle\mathfrak L\rangle$ in $\bP_\rA$ is isomorphic to $F$.
\end{theorem}

\begin{definition}
A tuple $(i_1,\dots,i_k)$ in $N$ is \textit{admissible} if its elements are pairwise distinct and for every $l\in[1,n]$ the number of elements with $|i_j|\le l$ does not exceed $l$. Let $\Theta$ denote the set of all admissible tuples of the form $(i_1\lessdot\dots\lessdot i_k)$ and $\Theta'$ denote the set of all non-admissible tuples of the same form.
\end{definition}

\begin{theorem}\label{intermediate}
There exists a monomial order $\ll$ on $S$ such that 
\begin{enumerate}[label=(\alph*)]
\item $\initial_\ll \mathfrak L=\spann(X_{i_1,\dots,i_k})_{(i_1,\dots,i_k)\in\Theta'}$ and
\item $\initial_\ll (I_\rA+\langle\mathfrak L\rangle)=I_\rA+\langle X_{i_1,\dots,i_k}\rangle_{(i_1,\dots,i_k)\in\Theta'}$.
\end{enumerate}
\end{theorem}
\begin{proof}
Part (a) is essentially proved in~\cite{Ma}, let us introduce the necessary notation. Consider the lattice $\bZ^N$ with basis $\{\zeta_{j}\}_{j\in N}$.
We define a $\bZ^N$-grading on $S$ (which should be thought of as the $GL(\bC^N)$-weight) by setting
\[\wt_\rA(X_{i_1,\dots,i_k})=\zeta_{i_1}+\dots+\zeta_{i_{k}}.\]

We make use of an alternative order on the set $N$: \[-1\lessdot' 1\lessdot'\dots\lessdot'-n\lessdot' n.\]
Now, consider the lexicographic order $\ll$ on $\bZ^N$ such that $\sum_{i\in N} a_i\zeta_i \ll \sum_{i\in N} b_i\zeta_i$
if for some $i \in N$ one has $a_{i}<b_{i}$ and $a_j=b_j$ for all $j\lessdot'i$. We extend this order to monomials in $S$ by setting $M_1\ll M_2$ if and only if $\wt_\rA M_1\ll \wt_\rA M_2$. It is not hard to check that the restriction of this order to the set of variables is inverse to the order defined in \cite[Definition 1.4]{Ma}. Therefore, using \cite[Lemma 1.16]{Ma} we obtain part~(a).

By part (a) we have
\[\initial_\ll (I_\rA+\langle\mathfrak L\rangle)\supset I_\rA+\langle X_{i_1,\dots,i_k}\rangle_{(i_1,\dots,i_k)\in\Theta'}.\] We show that we have a reverse inequality between graded dimensions of the ideals which provides part (b).

Let $T^\rA$ be the set of \textit{standard monomials}
$\prod_{l=1}^m X_{i^l_1,\dots,i_{k_l}^l}$ for which the $m$ tuples $(i^l_1,\dots,i^l_{k_l})$ are the columns of a semistandard Young tableau. This means that $k_l\geq k_{l+1}$ and $i_j^l\ledot' i_j^{l+1}$ for all $l\le m-1$ and $j\le k_{l+1}$ (the elements increase non-strictly within rows) and $i^l_j\lessdot' i^l_{j+1}$ (the elements increase strictly within columns). Let $T^\mathrm{C}\subset T^\rA$ consist of those monomials for which all $(i_1^l,\dots,i_{k_l}^l)$ are admissible (i.e.\ the tableau is \textit{symplectic semistandard}).

It is a classical fact that $T^\rA$ maps to a basis in $S/I_\rA$. This implies that the image of $T^\mathrm{C}$ spans $S/(I_\rA+\langle X_{i_1,\dots,i_k}\rangle_{(i_1,\dots,i_k)\in\Theta'})$. However, by \cite[Theorem 3.1]{DC} the subset $T^\mathrm{C}$ also maps to a basis in $S/(I_\rA+\langle\mathfrak L\rangle)$. We obtain the desired inequality.
\end{proof}

Denote $\bC[\Theta]=\bC[X_{i_1,\dots,i_k}]_{(i_1,\dots,i_k)\in\Theta}$. We define a homomorphism $\rho_1:S\to\bC[\Theta]$ as follows. Theorem~\ref{intermediate}(a) shows that for every $(i_1,\dots,i_k)\in\Theta'$ there exists a unique $R\in\spann(X_{i_1,\dots,i_k})_{(i_1,\dots,i_k)\in\Theta}$ such that $X_{i_1,\dots,i_k}-R\in\mathfrak L$. We set $\rho_1(X_{i_1,\dots,i_k})=R$. For $(i_1,\dots,i_k)\in\Theta$ we set $\rho_1(X_{i_1,\dots,i_k})=X_{i_1,\dots,i_k}$. The kernel of $\rho_1$ is $\langle\mathfrak L\rangle$. 

Next, for $k\in[1,n]$ let $\Theta_k$ denote the set of $k$-element tuples in $\Theta$. Let $I$ denote the image $\rho_1(I_\rA+\langle\mathfrak L\rangle)=\rho_1(I_\rA)$. Then the zero set of $I$ in \[\bP=\bP(\bC^{\Theta_1})\times\dots\times\bP(\bC^{\Theta_n})\] is again the complete flag variety. In other words, $\rho_1$ corresponds to an embedding of $\bP$ into $\bP_\rA$ and the image of this embedding contains the symplectic flag variety. The quotient $\bC[\Theta]/I= S/(I_\rA+\langle\mathcal L\rangle)$ is the \textit{symplectic Pl\"ucker algebra}.

Now, let $\rho_0:S\to\bC[\Theta]$ be the homomorphism taking every $X_{i_1,\dots,i_k}$ with $(i_1,\dots,i_k)\in\Theta$ to itself and all $X_{i_1,\dots,i_k}$ with $(i_1,\dots,i_k)\in\Theta'$ to zero. Let $\widetilde I$ be the image $\rho_0(I_\rA)$. A version of Theorem~\ref{intermediate}(b) is the following.
\begin{cor}\label{tildeinitial}
$\widetilde I$ is an initial ideal of $I$.    
\end{cor}
\begin{proof}
The monomial order $\ll$ defined in the proof of Theorem~\ref{intermediate} can be restricted to $\bC[\Theta]$, we show that $\initial_\ll I=\widetilde I$. We have seen that the set $T^\mathrm{C}$ projects into a basis in both $S/(I_\rA+\al\mathfrak L\ar)=\bC[\Theta]/I$ and $S/(I_\rA+\al X_{i_1,\dots,i_k}\ar_{(i_1,\dots,i_k)\in\Theta'})=\bC[\Theta]/\widetilde I$. Hence, the ideals have the same graded dimensions and it suffices to show that $\widetilde I\subset \initial_\ll I$.

Indeed, the ideals $I_\rA$ and $\widetilde I$ are $\wt_\rA$-homogeneous. For a nonzero $\wt_\rA$-homogeneous $p\in\widetilde I$ we have $p=\rho_0(q)$ for some $\wt_\rA$-homogeneous $q\in I_\rA$ and $q=\initial_\ll q'$ for some $q'\in I_\rA+\al\mathfrak L\ar$. Since $\rho_1$ replaces every variable $X_{i_1,\dots,i_k}\notin\bC[\Theta]$ occurring in $q'$ with a linear combination of $\ll$-smaller variables, we have $\initial_\ll\rho_1(q')=p$.
\end{proof}

We will construct flat degenerations of the flag variety by considering initial ideals of $\widetilde I$ which will then be initial ideals of $I$ by Corollary~\ref{tildeinitial}. The advantage of $\widetilde I$ is that its initial ideals are especially convenient to obtain in the language of sagbi degenerations in view of the following. Let us define a homomorphism $\varphi:\bC[\Theta]\to\bC[P]$. For $(i_1,\dots,i_k)\in\Theta$ let $Z(i_1,\dots,i_k)$ be the matrix with rows $1,\dots,k$ and columns $i_1,\dots,i_k$ such that $Z(i_1,\dots,i_k)_{a,b}=z_{a,b}$ if $|b|\ge a$ and $Z(i_1,\dots,i_k)_{a,b}=0$ otherwise. Set \[\varphi(X_{i_1,\dots,i_k})=D_{i_1,\dots,i_k}=|Z(i_1,\dots,i_k)|.\] In other words, $D_{i_1,\dots,i_k}$ is obtained from $C_{i_1,\dots,i_k}$ by setting all $z_{i,j}$ with $(i,j)\notin P$ to zero. We denote the image $\varphi(\bC[\Theta])\subset\bC[P]$ by $\mathcal{R}$.

\begin{proposition}
$\widetilde I$ is the kernel of $\varphi$.    
\end{proposition}
\begin{proof}
Let $F_\rA$ denote the partial flag variety of $GL(\bC^N)$ of signature $(1,\dots,n)$ cut out by the ideal $I_\rA$. The space $\bC^N$ is spanned by $\{e_i\}_{i\in N}$ and every ordering of this basis defines a Borel subgroup in $GL(\bC^N)$ consisting of elements whose matrices in this ordered basis are \textbf{lower} triangular. We consider the Borel subgroup $B_\rA$ given by the ordering $e_{-1},e_1,\dots,e_{-n},e_n$, i.e.\ with respect to $\lessdot'$.

To prove the proposition we consider Schubert varieties in $F_\rA$ with respect to the Borel subgroup $B_\rA$. Recall that such a variety is the orbit closure $\overline{B_\rA x}$ where $x\in F_\rA$ is a $(\bC^*)^N$-fixed point. Such a point is given by a tuple $\mathcal I=(i_1,\dots,i_n)$ in $N$, it is the point $x^{\mathcal I}\in\bP_\rA$ with multihomogeneous coordinates $x^{\mathcal I}_{i_1,\dots,i_k}=1$ and all other coordinates 0. Denote the Schubert variety $\overline{B_\rA x^{\mathcal I}}$ by $\mathcal X_{\mathcal I}$. The vanishing ideal of $\mathcal X_{\mathcal I}$ has a standard description: it is generated by $I_\rA$ together with all $X_{l_1,\dots,l_k}$ such that $l_1\lessdot'\dots\lessdot' l_k$ and for at least one $r\in[1,k]$ we have $l_r\lessdot' i_r$. In particular, the vanishing ideal of $\mathcal X_{(-1,\dots,-n)}$ is $I_\rA+\langle X_{i_1,\dots,i_k}\rangle_{(i_1,\dots,i_k)\in\Theta'}$. Hence, $\mathcal X_{(-1,\dots,-n)}$ is contained in the subspace $\bP\subset\bP_\rA$ where it is cut out by $\widetilde I\subset\bC[\Theta]$. We now use the fact that the vanishing ideal of a Schubert variety can alternatively be characterized as a kernel.

Consider the projection $\pi:GL(\bC^N)\to F_\rA$ mapping $g$ to the flag \[g\spann(e_{1})\subset\dots\subset g\spann(e_{1},\dots,e_{n}).\] Let $M\in\bC^{N\times N}$ be the matrix of $g$ in the basis $\{e_i\}_{i\in N}$, then the multihomogeneous coordinate $\pi(g)_{j_1,\dots,j_k}$ is equal to the minor of $M$ spanned by rows $j_1,\dots,j_k$ and columns $1,\dots,k$. Now consider the element $w\in GL(\bC^N)$ mapping $e_j$ to $e_{-j}$. We have $\pi(w)=x^{(-1,\dots,-n)}$. Consequently, $\overline{\pi(B_\rA w)}=\mathcal X_{(-1,\dots,-n)}$. 

Now, let $Y\subset\bC^{2n\times n}$ denote the set of $2n\times n$ matrices of rank $n$. We can write $\pi=\pi_2\pi_1$ where $\pi_1:GL(\bC^N)\to Y$ forgets columns $-1,\dots,-n$ and $\pi_2:Y\to F_A$ maps a matrix to the flag spanned by its columns. 
However, $\pi_1(B_\rA w)\subset Y$ consists of those $2n\times n$ matrices $M$ for which $M_{j,i}=0$ for $(i,j)\notin P$. Also, for a matrix $M\in Y$ the homogeneous coordinate $\pi_2(M)_{i_1,\dots,i_k}$ is equal to $D_{i_1,\dots,i_k}|_{z_{i,j}=M_{j,i}}$ if $(i_1,\dots,i_k)\in\Theta$ and to $0$ otherwise. Thus, the Schubert variety $\overline{\pi_2\pi_1(B_\rA w)}$ is cut out in $\bP$ by the kernel of $\varphi$.
\end{proof}

\begin{cor}\label{schubert}
The zero set of $\widetilde I$ in $\bP$ is the Schubert variety $\mathcal X_{(-1,\dots,-n)}$, it is isomorphic to $\MultiProj \mathcal{R}$. In particular, $\mathcal X_{(-1,\dots,-n)}$ is a flat degeneration of $F$.
\end{cor}

\subsection{Toric degenerations}

The following key property of combinatorial type C pipe dreams will let us define an isomorphism between $\bC[\cJ]$ and $\bC[\Theta]$ providing the toric degeneration.

\begin{lemma}
For $J\in\cJ_k$ the tuple $(w^{O,J}(1),\dots,w^{O,J}(k))$ is admissible.
\end{lemma}
\begin{proof}
This is immediate from Lemma~\ref{wrproperties}\ref{rijgei}
\end{proof}

The key ingredient of our first main result is the following map.
\begin{definition}
Let $\psi:\bC[\cJ]\to\bC[\Theta]$ be the homomorphism such that for $J\in\cJ_k$: \[\psi(X_J)=X_{w^{O,J}(1),\dots,w^{O,J}(k)}.\]
\end{definition}

\begin{example}
Consider $O$ and $J$ as in Example~\ref{xiex}. We have 
$\psi(X_J)=X_{-2,3}$.
\end{example}
The map $\psi$ encodes a correspondence $J\mapsto (w^{O,J}(1),\dots,w^{O,J}(k))$ between order ideals and admissible tuples. For the first claim in the below theorem ($\psi$ is an isomorphism) we will show that this is, in fact, a bijection between order ideals and \textit{unordered} admissible tuples. By this we mean that for every admissible tuple there is exactly one order ideal corresponding to a permutation of this tuple. 
\begin{theorem}\label{degenmainC}
The map $\psi$ is an isomorphism. Furthermore, the image $\psi(I_O)$ is an initial ideal of $I$. Hence, the toric variety of the polytope $\cQ_O(\la)$ with regular $\la$ is a flat degeneration of $F$.
\end{theorem}


\begin{example}\label{psiexgt}
Consider the case $O=P$. One has $M_P(J)=J$ and one sees that for $J\in\cJ_k$ the tuple $(w_J(1),\dots,w_J(k))$ decreases with respect to $\lessdot$. We also have \[w_P^{-1}(1,\dots,n,-n,\dots,-1)=(-n,n,\dots,-1,1).\] This means that the tuples $(w^{P,J}(1),\dots,w^{P,J}(k))$ are precisely the admissible subsequences of $(1,-1,\dots,n,-n)$. Furthermore, the ideal $I_P$ is the \textit{Hibi ideal} of the distributive lattice $\cJ\bs\{\varnothing\}$: it is generated by binomials of the form $X_{J_1}X_{J_2}-X_{J_1\cap J_2}X_{J_1\cup J_2}$. The fact that the image of such an ideal under $\psi$ is an initial ideal of $I$ can be viewed as a type C analog of the classical result in~\cite{GL}. This toric degeneration was first obtained in~\cite{Ca} (in rather different terms).
\end{example}

\begin{example}\label{psiexfflv}
Consider $O=A$. One has $M_A(J)=\max_\prec J\cup(A\cap J)$ while $w_A=\id$. One may check that the tuples \[(w^{A,J}(1),\dots,w^{A,J}(k))=(w_{M_A(J)}(1),\dots,w_{M_A(J)}(k))\] with $J\in\cJ_k$ are precisely those which have the following form. These tuples are admissible and any element $i\ledot k$ of the tuple must be in position $i$ while the elements $i\gtrdot k$ are arranged in decreasing order with respect to $\lessdot$. In the terminology of~\cite{Ba} these are the tuples forming one-column \textit{symplectic PBW tableaux}. The defining ideal $I_O$ of the toric variety of the type C FFLV polytope has a more complicated explicit description than in the previous example. The fact that the toric variety of the type C FFLV polytope is a flat degeneration of $F$ is due to~\cite{FFFM}, see also~\cite{BallaFang}.
\end{example}

To prove the second claim in Theorem~\ref{degenmainC} we will show that for a certain monomial order $<$ on $\bC[P]$ the algebra $\varphi_O(\bC[\cJ])\simeq\bC[\Theta]/\psi(I_O)$ is equal to $\initial_<\cR$. Since $\cR=\varphi(C[\Theta])$, this will let us apply Proposition~\ref{idealsubalg}. We now define the order $<$.

\begin{definition}\label{monorderC}
We introduce a lexicographic order on $\bZ^P$, to do so we first define a total order on $P$. Every $(i,j)\in P$ can be uniquely expressed as $(i,r(i,j'))$. We set $(i_1,r(i_1,j_1))<(i_2,r(i_2,j_2))$ if $i_1<i_2$ or ($i_1=i_2$ and $j_1\lessdot j_2$) with one exception.
If $i_1=i_2$, $j_1\lessdot j_2$, $(i_1,j_1)\in O$ and there is no $j_1\lessdot j'\ledot j_2$ with $(i_1,j')\in O$, we set $(i_1,r(i_1,j_1))>(i_2,r(i_2,j_2))$. 
Now, for distinct $d,d'\in\bZ^P$ we set $d<d'$ if for the $<$-minimal $(i,j)$ such that $d_{i,j}\neq d'_{i,j}$ we have $d_{i,j}>d'_{i,j}$. We also view $<$ as a total monomial order on $\bC[P]$ by setting $z^d<z^{d'}$ if $d<d'$. 
\end{definition}

One may alternatively describe the total order defined on the set of $(i,j)\in P$ with a chosen $i$ as follows. Let $a_1\lessdot\dots\lessdot a_l$ be those elements for which $(i,a_j)\in O$. First consider the total order $<'$ for which $(i,j_1)<'(i,j_2)$ when $j_1\lessdot j_2$ unless $j_1=a_k$ for some $k$ and $a_k\lessdot j_2\lessdot a_{k+1}$ (if $k=l$, the second inequality is omitted). To define $<$ we reorder the elements via $r$ by setting $(i,r(i,j_1))<(i,r(i,j_2))$ if and only if $(i,j_1)<'(i,j_2)$. For instance, the $<'$-maximal $(i,j)$ is $(i,a_l)$, hence, the $<$-maximal $(i,j)$ is $(i,r(i,a_l))=(i,i)$.

Set $D_{i_1,\dots,i_k}=(-1)^\sigma D_{i_{\sigma(1)},\dots,i_{\sigma(k)}}$ for $\sigma\in \mathcal S_k$ and any admissible tuple $(i_1,\dots,i_k)$.

\begin{proposition}\label{initialD}
For $J\in\cJ_k$ one has \[\initial_< D_{w^{O,J}(1),\dots,w^{O,J}(k)}=z^{\xi\left(\mathbf 1_{M_O(J)}\right)}=\prod_{i=1}^k z_{i,w^{O,J}(i)}.\]    
\end{proposition}
\begin{proof}
It suffices to show for every $i\in[1,k]$ that $z_{i,w^{O,J}(i)}$ is $<$-maximal among variables of the form $z_{i,w^{O,J}(j)}$ with $j\in[i,k]$. Recall that $w^{O,J}(i)=r(i,l)$ for the $\lessdot$-maximal $l$ such that $(i,l)\in M_O(J)$. Suppose that $z_{i,w^{O,J}(j)}>z_{i,w^{O,J}(i)}$ for some $j\in[i+1,k]$. Consider the $j$th pipe of $M_O(J)$, it passes through some $(i,l')$. Consider the $\lessdot$-minimal such $l'$, then $r(i,l')=w^{O,J}(j)$. Note that before passing through $(i,l')$ the $j$th pipe of $M_O(J)$ turns at least once and, since it only turns at elements of $J$, we obtain $(i,l')\in J$. 

If $l'\gtrdot l$, then $(i,j')\notin O$ for all $l\lessdot j'\ledot l'$. That is since all such $(i,j')\in J$ and $(i,j')\in O$ would contradict our choice of $l$. Since $(i,l)\in J\cap O$, we are in the exceptional case in Definition~\ref{monorderC} and obtain $z_{i,r(i,l')}<z_{i,r(i,l)}$ contradicting our assumption. 

Now suppose that $l'\lessdot l$. 
If $(i,l')\notin O$, we again have $z_{i,r(i,l')}<z_{i,r(i,l)}$. 
If $(i,l')\in O$, then the $j$th pipe of $M_O(J)$ turns at $(i,l')$ and prior to that it  turns at some other $(i,l'')$. In other words, there must exist $(i,l'')\in M_O(J)$ with $l''\gtrdot l'$ which is passed by this pipe prior to $(i,l')$. Note that, since the pipe necessarily turns at least once prior to $(i,l'')$, we may not have $(i,l'')\in\max_\prec J$, hence $(i,l'')\in J\cap O$. Consequently, we also have $l''\ledot l$. Hence, $l'\lessdot l''\ledot l$ for some $(i,l'')\in O$ and again $z_{i,r(i,l')}<z_{i,r(i,l)}$.    
\end{proof}

Note that the algebra $\bC[P]$ is graded by the group of integral weights: $\grad z_{i,j}=\varepsilon_i$. In particular, $\grad D_{i_1,\dots,i_k}=\om_k$ and the subalgebra $\cR$ is graded by the semigroup of integral dominant weights. The ring $\bC[\Theta]$ is graded by the same semigroup with $\grad X_{i_1,\dots,i_k}=\om_k$, the ideals $I,\widetilde I$ and the map $\varphi$ are $\grad$-homogeneous. For $U$ a $\grad$-homogeneous subspace or quotient of $\bC[P]$ or $\bC[\Theta]$ and integral weight $\la$ we denote the respective homogeneous component by $U[\la]$. The component $(\bC[\Theta]/I)[\la]$ of the Pl\"ucker algebra is known to have dimension $\dim V_\la$. Proposition~\ref{tildeinitial} the provides $\dim \cR[\la]=\dim V_\la$. 

\begin{proof}[Proof of Theorem~\ref{degenmainC}]
Note that the points $\mathbf 1_{M_O(J)}$ with $J\in\cJ$ are pairwise distinct, hence the $z^{\xi\left(\mathbf 1_{M_O(J)}\right)}$ are pairwise distinct monomials. Proposition~\ref{initialD} then implies that the sets $\{w^{O,J}(1),\dots,w^{O,J}(k)\}$ with $J\in\cJ_k$ are pairwise distinct. Since $|\cJ_k|=|\Theta_k|$, we deduce that every variable $X_{i_1,\dots,i_k}\in\bC[\Theta]$ is equal to $\pm X_{w^{O,J}(1),\dots,w^{O,J}(k)}$ for a unique $J\in\cJ_k$. This provides the first claim.

Now, Proposition~\ref{initialD} also implies that every monomial $z^{\xi\left(\mathbf 1_{M_O(J)}\right)}$ is contained in $\initial_<\cR$, i.e.\ $\varphi_O(\bC[\cJ])\subset\initial_<\cR$. However, \[\dim\varphi_O(\bC[\cJ])[\la]=|\cQ_O(\la)\cap\bZ^P|=\dim V_\la=\dim\cR[\la]\] and we deduce that $\varphi_O(\bC[\cJ])=\initial_<\cR$. Moreover, we see that the determinants $D_{i_1,\dots,i_k}$ with $(i_1,\dots,i_k)\in\Theta$ form a sagbi basis of $\cR$ with respect to $<$. In the notations used in Proposition~\ref{idealsubalg} we have $\varphi_<=\varphi_O\psi^{-1}$ and the proposition provides \[\initial_{<^\varphi}\widetilde I=\ker(\varphi_O\psi^{-1})=\psi(I_O).\]
By Proposition~\ref{tildeinitial} and transitivity $\psi(I_O)$ is also an initial ideal of $I$. 

The last claim is, of course, an application of Theorem~\ref{flatfamily}.
\end{proof}

To conclude this subsection we briefly discuss the connection to standard monomial theories. Lemma~\ref{pointdecomp} shows that the set of products $X_{J_1}\dots X_{J_m}$ with $J_1\subset\dots\subset J_m$ projects to a basis in $\bC[\cJ]/I_O$. Hence, the set of products $\psi(X_{J_1})\dots \psi(X_{J_m})$ with $J_1\subset\dots\subset J_m$ projects to a basis in $\bC[\Theta]/\psi(I_O)$. The fact that $\psi(I_O)$ is an initial ideal of $I$ then provides the following. 
\begin{cor}\label{standmon}
The set of all products $\psi(X_{J_1})\dots \psi(X_{J_m})$ with $J_1\subset\dots\subset J_m$ projects to a basis in the symplectic Pl\"ucker algebra $\bC[\Theta]/I$. 
\end{cor}
\begin{remark}\label{tableaux}
Bases of the above form are known as \textit{standard monomial theories}. In the case $O=P$ this basis is de Concini's basis parametrized by symplectic semistandard tableaux considered above as $T^\rC$ (modulo swapping $i$ and $-i$ which actually coincides with the conventions in~\cite{DC}). In the case $O=A$ we obtain the basis parametrized by symplectic PBW-semistandard tableaux described in~\cite{Ba}. 

In general, with every monomial $\psi(X_{J_1})\dots \psi(X_{J_m})$ in our basis we can associate a Young tableaux whose $(m+1-i)$th column contains the elements $w^{O,J_i}(1),\dots,w^{O,J_i}(k)$ where $J_i\in\cJ_k$ (we use English notation). We can then declare the resulting tableaux to be (semi)standard so that our basis is enumerated by such tableaux. Of course, in the cases $O=P$ and $O=A$ we recover the aforementioned families of tableaux.
\end{remark}

\subsection{PBW-monomial bases}

For $(i,j)\in P\bs A$ let $f_{i,j}$ denote the root vector in $\fg$ corresponding to the negative root $-\alpha_{i,j}$. Consider the space $V=\bC^N$ with basis $\{e_i\}_{i\in N}$. The Lie algebra $\fgl(V)$ consists of matrices with rows and columns indexed by $N$. Let $E_{i,j}\in\fgl(V)$ denote the matrix with the element in row $i$ and column $j$ equal to 1 and all other elements 0. Recall that the symplectic Lie algebra $\fg$ is standardly identified with a subalgebra of $\fgl(V)$ as follows: 
\[
f_{i,j}=
\begin{cases}
E_{j,i}-E_{-i,-j}&\text{ if }j\in[i+1,n],\\
E_{j,i}+E_{-i,-j}&\text{ if }j\in[-n,-i].
\end{cases}
\]

This realization allows us to view the space $\wedge^k V$ as a $\fg$-representation. We denote elements of the multivector basis in $\wedge^k V$ by $e_{i_1,\dots,i_k}=e_{i_1}\wedge\dots\wedge e_{i_k}$. We identify the fundamental representation $V_{\om_k}$ with the subspace $\cU(\fg)(e_{1,\dots,k})\subset\wedge^kV$ and assume that $v_{\om_k}=e_{1,\dots,k}$. 
We also recall how the matrix $E_{a,b}$ acts on multivectors. If $\{i_1,\dots,i_k\}$ contains $b$ but not $a$, then $E_{a,b}$ maps $e_{i_1,\dots,i_k}$ to $e_{j_1,\dots,j_k}$ where $j_1,\dots,j_k$ is obtained from $i_1,\dots,i_k$ by replacing $b$ with $a$. Otherwise $E_{a,b}e_{i_1,\dots,i_k}=0$. 

For $d\in\bZ_{\ge 0}^P$ we use the notation \[f^d=\prod_{(i,j)\in P\bs A} f_{i,j}^{d_{i,j}}\] where the factors are ordered first by $i$ increasing from left to right and then by $j$ increasing with respect to $\lessdot$ from left to right.

The ring $\bC[P]$ is also graded by $GL(V)$-weights by setting $\wt_\rA(z_{i,j})=\zeta_j-\zeta_i$. 
Recall the order $\ll$ on $\bZ^N$ (page~\pageref{intermediate}).
We say that a monomial order $<'$ on $\bC[P]$ is \textit{$\ll$-compatible} if for any two monomials $M_1$, $M_2$ such that $\wt_\rA(M_1)\ll\wt_\rA(M_2)$ one has $M_1<'M_2$.

\begin{lemma}\label{fundamentalC}
Consider a $\ll$-compatible total monomial order $<'$ on $\bC[P]$ and $k\in[1,n]$. 
\begin{enumerate}[label=(\alph*)]
\item A basis in $V_{\om_k}$ is formed by the vectors $f^d v_{\om_k}$ such that $z^d=\pm\initial_{<'} D_{i_1,\dots,i_k}$ for some $(i_1,\dots,i_k)\in\Theta$.
\item For any $d'\in\bZ_{\ge 0}^P$ the decomposition of $f^{d'}v_{\om_k}$in the above basis only contains vectors $f^d v_{\om_k}$ for which $d'\le' d$.
\end{enumerate}
\end{lemma}
\begin{proof}
\hfill

(a) We have $\initial_{<'}D_{i_1,\dots,i_k}=\pm z_{1,i_{\sigma(1)}}\dots z_{k,i_{\sigma(k)}}$ for some $\sigma\in\mathcal S_k$. Consider the element
\[v_{i_1,\dots,i_k}=f_{1,i_{\sigma(1)}}\dots f_{k,i_{\sigma(k)}}v_{\omega_k}\in V_{\omega_k},\]
where factors of the form $f_{i,i}$ are omitted.
Recall that $V_{\omega_k}\subset \wedge^kV$. 
Let us extend the order $\ll$ to the set of multivectors:
$e_{j_1,\dots,j_k}\ll e_{j'_1,\dots,j'_k}$
if $
\wt_\rA e_{j_1,\dots,j_k}\ll \wt_\rA e_{j'_1,\dots,j'_k}
$ 
where \[\wt_\rA e_{j_1,\dots,j_k}=
\zeta_{j_1}+\dots+\zeta_{j_k}-\zeta_1-\dots-\zeta_k.\]
Then we have 
\[v_{i_1,\dots,i_k}\in c e_{i_{\sigma(1)},\dots,i_{\sigma(k)}}+\bigoplus_{e_{j_1,\dots,j_k}\gg e_{i_{\sigma(1)}, \dots,i_{\sigma(k)}}}\mathbb{C}e_{j_1,\dots, j_k}\]
where $c\neq 0$. Indeed, if $i_{\sigma(j)}\neq -j$ for some $j\in[1,k]$, then $f_{j,i_{\sigma(j)}}=E_{i_{\sigma(j)},j}\pm E_{-j,-i_{\sigma(j)}}$ and if $i_{\sigma(j)}=-j$, then $f_{j,i_{\sigma(j)}}=2E_{i_{\sigma(j)},j}$. Hence, the product $f_{1,i_{\sigma(1)}}\dots f_{k,i_{\sigma(k)}}$ expands into a linear combination of products of the elements $E_{j,i}$. One of these products is $E_{i_{\sigma(1)},1}\dots E_{i_{\sigma(k)},k}$, it occurs with coefficient $c=\pm2^{|\{j|i_{\sigma(j)}=-j\}|}$. All others either act on $v_{\om_k}$ trivially or map it into a multivector $e_{j_1,\dots,j_k}\gg e_{i_{\sigma(1)}, \dots,i_{\sigma(k)}}$. Indeed, applying $E_{i_{\sigma(j)},j}$ to a multivector adds $\delta_1=\zeta_{i_{\sigma(j)}}-\zeta_j$ to its $\wt_\rA$-grading while applying $E_{-j,-i_{\sigma(j)}}$ adds $\delta_2=\zeta_{-j}-\zeta_{-i_{\sigma(j)}}$. However, $\delta_2\gg\delta_1$ since $j$ is $\lessdot'$-minimal among $-j$, $j$, $-i_{\sigma(j)}$ and $i_{\sigma(j)}$.

Hence, the $\ll$-minimal multivectors occurring in the decompositions of the various $v_{i_1,\dots,i_k}$ are pairwise distinct. Thus, these vectors are linearly independent. The number of such vectors coincides with $\dim V_{\omega_k}$ and we obtain part (a).

(b) Similarly, in $f^{d'}$ expand every $f_{i,j}$ with $j\neq-i$ as $E_{j,i}\pm E_{-i,-j}$ and express $f^{d'}v_{\om_k}$ as a linear combination of products of the $E_{i,j}$ applied to $v_{\om_k}$. Each summand is a scalar multiple of a multivector. 
Similarly to part (a) one shows any multivector $e_{j_1,\dots,j_k}$ occurring in this sum satisfies the non-strict inequality 
\begin{equation}\label{weightineq}
\wt_A e_{j_1,\dots,j_k} \ge\!\!\!\ge \wt_\rA z^{d'}=\sum d'_{i,j}(\zeta_j-\zeta_i).    
\end{equation}
We have seen that for $(j_1,\dots,j_k)\in\Theta$ the $\ll$-minimal multivector occurring in $v_{j_1,\dots,j_k}$ is $e_{j_1,\dots,j_k}$. Hence, the decomposition of $f^{d'}v_{\om_k}$ in the basis found in part (a) may not contain $v_{j_1,\dots,j_k}$ with $\wt_\rA e_{j_1,\dots,j_k}\ll \wt_\rA z^{d'}$. Let $v_{j_1,\dots,j_k}$ occur in this decomposition, denote $\initial_{<'}D_{j_1,\dots,j_k}=\pm z^d$, note that $\wt_\rA z^d=\wt_\rA e_{j_1,\dots,j_k}$. The $\ll$-compatibility of the order implies that either $d'<' d$ or $\wt_\rA z^{d'}=\wt_\rA z^d$. Since $e_{j_1,\dots,j_k}$ occurs in  $f^{d'}v_{\om_k}$ with a nonzero coordinate, $\wt_\rA z^{d'}=\wt_\rA z^d$ is only possible if equality holds in~\eqref{weightineq}, i.e.\ $e_{j_1,\dots,j_k}=\pm\prod E_{j,i}^{d'_{i,j}}v_{\omega_k}$. However, since the factors in $\prod E_{j,i}^{d'_{i,j}}$ are ordered by $i$ increasing from left to right, it can only act nontrivially on $v_{\om_k}$ if $z^{d'}$ is a summand in the determinant $D_{j_1,\dots,j_k}$. We obtain $z^{d'}\le'z^d$.
\end{proof}

\begin{lemma}\label{sagbiToMonomials}
Suppose that for a total monomial order $<'$ on $\bC[P]$ the $D_{i_1,\dots,i_k}$ form a sagbi basis of $\mathcal{R}$ with respect to $<'$. Then for an integral dominant weight $\la$ the vectors $f^d v_\la$ with $z^d\in \initial_{<'} R[\la]$ form a basis in $V_\la$.
\end{lemma}
\begin{proof}
Suppose first that $<'$ is $\ll$-compatible. Let $\lambda=(a_1,\dots,a_n)$.
We have a standard embedding
\[V_\lambda \subset \bigotimes_{i=1}^n V_{\omega_i}^{\otimes a_i}=U_\la\] identifying $v_\la$ with $\bigotimes v_{\om_i}^{\otimes a_i}$.
Recall that in Lemma~\ref{fundamentalC}(a) we obtained a basis in $V_{\om_k}$ consisting of vectors $v_{i_1,\dots,i_k}$ with $(i_1,\dots,i_k)\in\Theta_k$. For every $v_{i_1,\dots,i_k}$ we have $\initial_{<'}D_{i_1,\dots,i_k}=\pm z^d$ for some $d\in\bZ^P$. We set $\deg v_{i_1,\dots,i_k}=d$ obtaining a $\bZ_{\ge 0}^P$-grading on every $V_{\om_k}$. We extend these gradings multiplicatively to $U_\la$, for $u\in U_\la$ we denote its $\deg$-homogeneous components by $u_d$, $d\in \bZ_{\ge 0}^P$.

Let us view $<'$ as an order on $\bZ_{\ge 0}^P$. We claim that for $z^d\in\initial_{<'}\cR[\la]$ the vector $u=f^dv_\la\in U_\la$ has $u_d\neq 0$ and $u_c=0$ for all $c\not\ge'd$. This will then provide the linear independence of such vectors and by $\dim\cR[\la]=\dim V_\la$ we obtain a basis. 

Indeed, expand \[u=f^d(v_{\om_1}^{\otimes a_1}\otimes\dots\otimes v_{\om_n}^{\otimes a_n})\] via the Leibniz rule into a sum over all ordered decompositions of $d$ into $a_1+\dots+a_n$ parts. By the sagbi basis assumption we have at least one decomposition \[d=\sum_{k=1}^n\sum_{i=1}^{a_k} d^k_i\] such that $z^{d^k_i}\in\initial_{<'}\cR[\om_k]$ for all $d^k_i$, i.e.\ $z^{d^k_i}$ has the form $\initial_{<'}\pm D_{i_1,\dots,i_k}$. The corresponding summand in our expansion of $u$ will just be a product of the respective $v_{i_1,\dots,i_k}$ and will have $\deg$-grading $d$. We might have multiple such summands if we have multiple decompositions of the above form but they will be pairwise distinct summands in $u_d$. 

Now consider a decomposition $d=\sum\sum d^k_i$ such that $z^{d^k_i}\notin\initial_{<'}\cR[\om_k]$ for at least one $d^k_i$. Then Lemma~\ref{fundamentalC}(b) implies that the corresponding summand \[u'=\bigotimes_{i=1}^{a_1}f^{d^1_i}v_{\om_1}\otimes\dots\otimes\bigotimes_{i=1}^{a_n}f^{d^n_j}v_{\om_n}\] has $u'_{c}\neq 0$ only for $c$ such that $d<'c$. Hence, $u_d\neq 0$ and $u_c=0$ for all $c\not\ge'd$.


Finally, relax the assumption that $<'$ is $\ll$-compatible. Consider the order $<'_\rA$ defined as follows:
$z^d<'_\rA z^{d'}$ if either $\wt_\rA z^d\ll\wt_\rA z^{d'}$ or $\wt_\rA z^d=\wt_\rA z^{d'}$ and $z^d<'z^{d'}$. The order $<'_\rA$ is $\ll$-compatible by construction. However, for $(i_1,\dots,i_k)\in\Theta$ one has \[\initial_{<'}D_{i_1,\dots,i_k}=\initial_{<'_\rA}D_{i_1,\dots,i_k}.\] The sagbi basis property then implies that $\initial_{<'}R[\la]=\initial_{<'_\rA}R[\la]$. Therefore, the claims of the lemma for $<'$ and $<'_\rA$ coincide.
\end{proof}

We now easily obtain the main result of this subsection.
\begin{theorem}\label{mainbasisC}
The vectors $f^x v_\la$ with $x\in \Pi_O(\la)\cap\bZ^P$ form a basis in $V_\la$.
\end{theorem}
\begin{proof}
Denote $\la=(a_1,\dots,a_n)$. Recall the order $<$ from Definition~\ref{monorderC}. We check that the space $\initial_<\cR[\la]$ has a basis formed by the monomials $z^x$ with $x\in\Pi_O(\la)\cap\bZ^P$. For $\la=\om_k$ this is by Proposition~\ref{initialD}. For general $\la$ the claim follows from two facts. The first is the Minkowski sum property of the  polytopes $\Pi_O(\la)$: Theorem~\ref{minkowskiC} evidently holds when $\cQ_O$ is replaced by $\Pi_O$. The second is that the $D_{i_1,\dots,i_k}$ form a sagbi basis in $\cR$ with respect to $<$, this was shown in the proof of Theorem~\ref{degenmainC}. Together these facts show that (1) the set of $z^x$ with $x\in\Pi_O(\la)\cap\bZ^P$ is the set of all products of $a_1+\dots+a_n$ monomials of which $a_k$ lie in $\initial_<\cR[\om_k]$ and (2) that the set of monomials in $\initial_<R[\la]$ is also described in this way. The theorem now follows directly from Lemma~\ref{sagbiToMonomials}.
\end{proof}

\begin{example}
As mentioned in Example~\ref{xiexfflv}, the polytopes $\Pi_A(\la)$ and $\cQ_A(\la)$ have the same projection to $\bR^{P\bs A}$ and this projection is unimodularly equivalent to both of them. This means that the basis $\{f^xv_\la|x\in\Pi_A(\la)\cap\bZ^P\}$ coincides with $\{f^xv_\la|x\in\cQ_A(\la)\cap\bZ^P\}$ and is the type C FFLV basis constructed in~\cite{FFL2}.
\end{example}

\begin{example}\label{basisexgt}
In view of Proposition~\ref{pointimage} and Example~\ref{psiexgt} the set $\Pi_P(\om_k)\cap\bZ^P$ consists of all points $\one_{\{(1,i_1),\dots,(k,i_k)\}}$ for which $(i_1,\dots,i_k)$ is an admissible subsequence of $(1,-1,\dots,n,-n)$. The set $\Pi_P(\la)\cap\bZ^P$ is then found as the Minkowski sum of the former sets. The authors are not aware of the resulting basis $\{f^xv_\la|x\in\Pi_P(\la)\cap\bZ^P\}$ appearing in the literature although a basis of similar structure is studied in~\cite{MY}.
\end{example}

\subsection{Newton--Okounkov bodies}\label{nosec}

We follow~\cite{KaKh,Ka} associating a Newton--Okounkov body of $F$ with a line bundle $\cL$, a global section $\tau$ of $\cL$ and a valuation $\nu$ on the function field $\bC(F)$. 

\begin{definition}\label{valdef}
For a total group order $<'$ on $\bZ^P$ a $(\bZ^P,<')$-valuation $\nu$ on $\bC(F)$ is a map $\nu:\bC(F)\bs\{0\}\to\bZ^P$ such that for $f,g\in\bC(F)\bs\{0\}$ and $c\in\bC^*$ one has
\begin{enumerate}
\item $\nu(fg)=\nu(f)+\nu(g)$,
\item $\nu(cf)=\nu(f)$,
\item if $f+g\neq0$, then $\nu(f+g)\le'\max_{<'}(\nu(f),\nu(g))$.
\end{enumerate}
\end{definition}
We note that in condition (3) it is, perhaps, more standard to dually require the valuation of the sum to be no less than the minimum. However, these two approaches differ only by reversing the total group order and the above is more convenient to us.

\begin{definition}
For a line bundle $\cL$ on $F$, a nonzero global section $\tau\in H^0(F,\cL)$, a total group order $<'$ on $\bZ^P$ and a $(\bZ^P,<')$-valuation $\nu$ on $\bC(F)$ the corresponding \textit{Newton--Okounkov body} of $F$ is the convex hull closure \[\Delta(\cL,\tau,<',\nu)=\overline{\conv\left\{\left.\frac{\nu(\sigma/\tau^{\otimes m})}m\right|m\in\bZ_{>0},\sigma\in H^0(F,\cL^{\otimes m})\bs\{0\}\right\}}\subset\bR^P.\]    
\end{definition}

We choose an integral dominant $\la=(a_1,\dots,a_n)$ and let $\cL$ be the $G$-equivariant line bundle on the flag variety associated with the weight $\la$. In terms of the embedding $F\subset\bP$ given by the ideal $I$ this bundle is the restriction of $\cO(a_1,\dots,a_n)$ to $F$. Hence, $H^0(F,\cL)$ is naturally isomorphic to $\cR^\rC[\la]$ where $\cR^\rC$ denotes the symplectic Pl\"ucker algebra $\bC[\Theta]/I$. We choose $\tau\in H^0(F_\la,\cL)$ as the image of the monomial $\prod_{k=1}^n X_{1,\dots,k}^{a_k}\in \bC[\Theta][\la]$ in $\cR^\rC$.

We consider the total group order on $\bZ^P$ denoted by $<_A$ in the proof of Lemma~\ref{sagbiToMonomials}: $x<_\rA y$ if either $\wt_\rA z^x\ll\wt_\rA z^y$ or $\wt_\rA z^x=\wt_\rA z^y$ and $x<y$. (Here $z^x\in\bC[z_{i,j}^{\pm1}]_{(i,j)\in P}$ and we consider the natural extension of $\wt_\rA$ to Laurent polynomials.)

To define the valuation $\nu$ we first define a valuation on $\cR^\rC$. Recall the homomorphism $\varphi_<=\varphi_O\psi^{-1}$ from $\bC[\Theta]$ to $\bC[P]$ (cf.\ proof of Theorem~\ref{degenmainC}) mapping the variable $\psi(X_J)$ to $z^{\xi(\one_{M_O(J)})}$.
We have a $(\bZ^P,<_\rA)$-filtration on $\bC[\Theta]$ with component $\bC[\Theta]_x$ for $x\in\bZ^P$ spanned by monomials $M$ with $\varphi_<(M)\le_\rA z^x$ (where we view $<_\rA$ as a monomial order on $\bC[P]$). This induces a filtration on $\cR^\rC$ with components $\cR^\rC_x=\bC[\Theta]_x/(I\cap \bC[\Theta]_x)$.
For an element $p\in\cR^\rC\bs\{0\}$ we set $\nu(p)$ to be the $<_\rA$-minimal $x$ for which $p\in\cR^\rC_x$. Since $\bC(F)$ consists of fractions $p/q$ where $p,q\in \cR[\mu]$ for some $\mu$, we can extend the valuation to $\bC(F)$ by $\nu(p/q)=\nu(p)-\nu(q)$.
\begin{lemma}
The map $\nu$ is a $(\bZ^P,<_\rA)$-valuation on $\bC(F)$.
\end{lemma}
\begin{proof}
It suffices to show that properties (1)--(3) from Definition~\ref{valdef} hold for the map $\nu:\cR^\rC\bs\{0\}\to \bZ^P$, since every valuation on an integral domain extends to a valuation on its field of fractions by the given formula. 

For a map $\nu$ obtained from a filtration $(\cR^\rC_x)_{x\in\bZ^P}$ in the above way property (1) is equivalent to the associated graded algebra $\gr\cR^\rC$ being an integral domain. Now, $\gr\cR^\rC$ is isomorphic to $\gr\bC[\Theta]/\gr I$ with respect to the filtration on $\bC[\Theta]$ its restriction to $I$. Here $\gr\bC[\Theta]$ is naturally isomorphic to $\bC[\Theta]$. Under this isomorphism $\gr I\subset\bC[\Theta]$ is seen to be the initial ideal $\initial_{<_\rA^{\varphi_<}} I$ where $<_\rA^{\varphi_<}$ is the pullback of $<_\rA$ to $\bC[\Theta]$ (cf.\ Proposition~\ref{idealsubalg}): for two monomials $M_1<_\rA^{\varphi_<} M_2$ if and only if $\varphi_<(M_1)<_\rA\varphi_<(M_2)$. However, by the definition of $<_\rA$, the order $<_\rA^{\varphi_<}$ amounts to comparing two monomials in $\bC[\Theta]$ first by $\ll$ (in the sense used in Corollary~\ref{tildeinitial}) and then by $<^{\varphi_<}=<^\varphi$. Hence, \[\initial_{<_\rA^{\varphi_<}} I=\initial_{<^\varphi}(\initial_\ll I)=\initial_{<^\varphi}\widetilde I=\psi(I_O)\] as seen in Corollary~\ref{tildeinitial} and Theorem~\ref{degenmainC}. We conclude that $\gr I=\psi(I_O)$ is prime and $\gr\cR^\rC$ is indeed an integral domain.

Finally, property (2) is obvious and property (3) is just the fact that the filtration $(\cR^\rC_x)_{x\in\bZ^P}$ is compatible with the order $<_\rA$.
\end{proof}

For every $k\in[1,n]$ we have a unique $J\in\cJ_k$ such that $w^{O,J}(i)=i$ for $i\in[1,k]$, denote $x_k=\one_{M_O(J)}$. For $\la=(a_1,\dots,a_n)$ denote $x_\la=a_1x_1+\dots+a_nx_n$. Note that $\xi(x_k)=\one_{\{(1,1),\dots,(k,k)\}}$, hence, $\xi(x_\la)_{i,i}=\la(i)$ while other of $\xi(x_\la)$ coordinates are 0.
\begin{theorem}\label{PiNOC}
$\Delta(\cL,\tau,<_\rA,\nu)=\Pi_O(\la)-\xi(x_\la)$.   
\end{theorem}
\begin{proof}
For $(i_1,\dots,i_k)\in\Theta$ let $Y_{i_1,\dots,i_k}$ denote the image of $X_{i_1,\dots,i_k}$ in $\cR^\rC$. Since the ideal $I$ is quadratically generated, any polynomial in $X_{i_1,\dots,i_k}+I$ contains the monomial $X_{i_1,\dots,i_k}$. Hence, the minimal $\cR^\rC_x$ containing $Y_{i_1,\dots,i_k}$ is the image of the minimal $\bC[\Theta]_x$ containing $X_{i_1,\dots,i_k}$. This provides \[z^{\nu(Y_{i_1,\dots,i_k})}=\varphi_<(X_{i_1,\dots,i_k})=z^{\xi(\one_{M_O(J)})}\] where $J$ is such that $X_J=\pm\psi^{-1}(X_{i_1,\dots,i_k})$. This implies that \[\nu(\cR^\rC[\om_k]\bs\{0\})\supset\nu(\{Y_{i_1,\dots,i_k}\}_{(i_1,\dots,i_k)\in\Theta_k})=\{\xi(\one_{M_O(J)})\}_{J\in\cJ_k}=\Pi_O(\om_k)\cap\bZ^P.\] By property (1) this extends to $\nu(\cR^\rC[\mu]\bs\{0\})\supset\Pi_O(\mu)\cap\bZ^P$ for any integral dominant $\mu$. However, a general property of valuations is that $|\nu(U\bs\{0\})|\le\dim U$ for any finite-dimensional subspace $U$. In view of $\dim\cR^\rC[\mu]=|\Pi_O(\mu)\cap\bZ^P|$ we obtain \[\nu(\cR^\rC[\mu]\bs\{0\})=\Pi_O(\mu)\cap\bZ^P.\]

We also have \[\nu(\tau)=\sum_{k=1}^n a_k\nu(Y_{1,\dots,k})=\sum_{k=1}^n a_k\xi(x_k)=\xi(x_\la).\] Note that $\cL^{\otimes m}$ is the line bundle associated with the weight $m\la$ so that $H^0(F,\cL^{\otimes m})=\cR^\rC[m\la]$ and that the section $\tau^{\otimes m}$ equals $\tau^m\in\cR^\rC[m\la]$. We now see that for any $m\in\bZ_{>0}$ we already have \[\conv\left\{\left.\frac{\nu(\sigma/\tau^{\otimes m})}m\right|\sigma\in H^0(F,\cL^{\otimes m})\bs\{0\}\right\}=\frac{\Pi_O(m\la)-\nu(\tau^m)}m=\Pi_O(\la)-\xi(x_\la).\qedhere\]
\end{proof}

It is now easy to deduce that $\cQ_O(\la)$ is (up to translation) also a Newton--Okounkov body of $F$. Indeed, consider the order $<_\rA^\xi$ on $\bZ[P]$ given by $x <_\rA^\xi y$ if $\xi(x) <_\rA \xi(y)$. For $p\in\bC(F)\bs\{0\}$ set $\nu^\xi(p)=\xi^{-1}(\nu(p))$. Evidently, $\nu^\xi$ is a $(\bZ^P,<_\rA^\xi)$-valuation and the following is a direct consequence of Theorem~\ref{PiNOC}.

\begin{theorem}\label{mainNOC}
$\Delta(\cL,\tau,<_\rA^\xi,\nu^\xi)=\cQ_O(\la)-x_\la$.
\end{theorem}

\section{Type B}

\subsection{Type B Lie algebras and representations}

In this section we consider the Lie algebra $\fg=\mathfrak{so}_{2n+1}(\bC)$. We use the notations $\fh$, $\alpha_i$, $\om_i$, $(a_1,\dots,a_n)\in\fh^*$, $V_\la$ and $v_\la$ similarly to type C. We also have a basis $\varepsilon_1,\dots,\varepsilon_n$ in $\fh^*$ such that $\alpha_i=\varepsilon_i-\varepsilon_{i+1}$ for $i\le n-1$ and $\alpha_n=\varepsilon_n$. The positive roots are again indexed by pairs of integers $i,j$ such that $i\in[1,n]$ and $j\in[i+1,n]\cup[-n,-i]$, i.e.\ $(i,j)\in P\bs A$. The root $\alpha_{i,j}$ is equal to $\varepsilon_i-\varepsilon_j$ when $j>0$ and to $\varepsilon_i+\varepsilon_j$ when $-i<j<0$ while $\alpha_{i,-i}=\varepsilon_i$. In particular, $\alpha_i=\alpha_{i,i+1}$ for $i\le n-1$, $\alpha_n=\alpha_{n,-n}$ and the roots $\alpha_{i,-i}$ are short. 

In the type B case one has $\om_i=\varepsilon_1+\dots+\varepsilon_i$ for $i\le n-1$ and $\om_n=(\varepsilon_1+\dots+\varepsilon_n)/2$. For a weight $\la=(a_1,\dots,a_n)$ we again denote its coordinates with respect to the basis $\varepsilon_1,\dots,\varepsilon_n$ by $(\la(1),\dots,\la(n))$. Explicitly: $\la(i)=a_i+\dots+a_n/2$.

Let $f_{i,j}$ denote the root vector corresponding to the negative root $-\alpha_{i,j}$. We will make use of the matrix realization of $\fg$. Let $V$ denote a $(2n+1)$-dimensional complex space with a basis enumerated by the set $[-n,n]$. The Lie algebra $\fgl(V)$ consists of matrices with rows and columns enumerated by $[-n,n]$, for $i,j\in[-n,n]$ let $E_{i,j}\in\fgl(V)$ denote the matrix with 1 at position $i,j$ and all other elements 0. Then $\fg$ can be identified with a subalgebra of $\fgl(V)$, one such identification (see~\cite[Section 8.3]{carter}) is given by
\begin{equation}\label{matrixform}
f_{i,j}=
\begin{cases}
E_{j,i}-E_{-i,-j}&\text{ if }|j|>i,\\
E_{0,i}-2E_{-i,0}&\text{ if }j=-i.
\end{cases}
\end{equation}

The above realization equips $V$ with a $\fg$-module structure. For $k\le n-1$ the corresponding fundamental representation is $V_{\om_k}=\wedge^k V$ with highest-weight vector $v_{\om_k}=e_1\wedge\dots\wedge e_k$ while $V_{\om_n}$ is the $2^n$-dimensional spin representation. One also has $V_{2\om_n}=\wedge^n V$ with highest-weight vector $v_{2\om_n}=e_1\wedge\dots\wedge e_n$.

\subsection{Type B poset polytopes}

We define another family of polytopes associated with the poset $(P,\prec)$. Fix a subset $O\subset P$ containing $A$ and not containing any elements from $B=\{(i,-i)\}_{i\in[1,n]}$.

\begin{definition}\label{polytopedefB}
For an integral dominant $\fg$-weight $\la$ the \textit{type B poset polytope} $\cQ^\rB_O(\la)\subset\bR^P$ consists of points $x$ such that:
\begin{itemize}
\item $x_{i,i}=\la(i)$ for all $1\le i\le n$,
\item all $x_{i,j}\ge 0$,
\item for every chain $(p,q)\prec(i_1,j_1)\prec\dots\prec(i_m,j_m)\prec(r,s)$ with $(p,q)\in O$, $(r,s)\in P$ and all $(i_l,j_l)\notin O$ one has \[x_{i_1,j_1}+\dots+x_{i_m,j_m}\le x_{p,q}-cx_{r,s}\]
where $c$ equals $1/2$ if $(r,s)\in B$ and 1 otherwise.
\end{itemize}    
\end{definition}

\begin{example}
In the case $O=P\bs B$ one obtains a Berenstein--Zelevinsky polytope (due to \cite{BZ}) which we refer to as the \textit{type B Gelfand--Tsetlin polytope}. In the case $O=A$ one obtains the polytope studied in~\cite{MtypeB} which we term the \textit{type B FFLV polytope}.
\end{example}

\begin{remark}
Note that Definition~\ref{polytopedefB} also makes sense when $O\cap B$ is nonempty. In fact, the requirement $O\cap B=\varnothing$ is not particularly restrictive because $\cQ^\rB_O(\la)$ does not change when an element of $B$ is added to or removed from $O$. Consequently, although the assumption $O\cap B=\varnothing$ allows for a nicer wording of some of the below results, they can, nonetheless, be generalized to the case of arbitrary $O\supset A$ using this observation.
\end{remark}

One sees that if $\la=(a_1,\dots,a_n)$ with $a_n$ even (i.e.\ $\la(n)$ is integer), then $\cQ^\rB_O(\la)$ is obtained from the type C poset polytope $\cQ_O((a_1,\dots,a_n/2))$ by scaling by a factor of 2 along the $n$ coordinates corresponding to $B$. 
Of course, this relation between type B and type C poset polytopes also holds when $a_n$ is odd if one generalizes Definition~\ref{mcopHdefC} verbatim to all (not just integral) dominant weights. Note, however, that in the latter case $\cQ^\rB_O(\la)$ will not be a lattice polytope. As a polytope in $\bR^P$ it will actually have no lattice points due to all $x_{i,i}$ being non-integers for $x\in\cQ^\rB_O(\la)$, however, other coordinates of vertices of $\cQ^\rB_O(\la)$ may be non-integers as well. For example, the vertices of $\cQ^\rB_O(\om_n)$ are the points $(\one_{M_O(J)}+\one_{J\cap B})/2$ with $J\in\cJ_n$.

The above explains why in the case of odd $a_n$ the point set relevant to us is not $\cQ^\rB_O(\la)\cap\bZ^P$ but the set of points lying in a different (affine) lattice, namely $\one_O/2+\bZ^P$. Our approach will be to consider a certain transformation of $\cQ^\rB_O(\la)$ under which points of this lattice are mapped to integer points. However, we will first discuss the lattice point set $\cQ^\rB_O(\la)\cap\bZ^P$ in the case of even $a_n$.

\begin{definition}
Consider $J\in\cJ$ and $D\subset[1,n]$ such that $(i,-i)\in J$ for all $i\in D$. Recall the set $M_O(J)$ (Definition~\ref{mcopdefC}). Let $x^{J,D}\in\bR^P$ denote the point with
\[
x^{J,D}_{i,j}=
\begin{cases}
0&\text{ if }(i,j)\notin M_O(J),\\
1&\text{ if }(i,j)\in M_O(J) \text{ and } ((i,j)\notin B \text{ or } i\in D),\\
2&\text{ if }(i,j)\in M_O(J)\cap B \text{ and }i\notin D.    
\end{cases}
\]
\end{definition}
In the above definition note that $(i,-i)\in M_O(J)$ if and only if $(i,-i)\in J$ and $x^{J,D}_{i,-i}$ equals 1 or 2 depending on whether $i\in D$.


\begin{lemma}\label{pointdecompB}
For an integral dominant $\la=(a_1,\dots,a_n)$ with $a_n$ even and a lattice point $x\in\cQ^\rB_O(\la)\cap\bZ^P$ there exist unique sequence of order ideals $J_1\subset\cdots\subset J_m$ in $\cJ$ and set $D\subset[1,n]$ such that \[x=x^{J_1,\varnothing}+\dots+x^{J_{m-1},\varnothing}+x^{J_m,D}\] and for $k\le n-1$ exactly $a_k$ of the $J_i$ lie in $\cJ_k$ while $a_n/2$ of the $J_i$ lie in $\cJ_n$.
\end{lemma}
\begin{proof}
First, note that if all $x_{i,-i}$ are even then the statement follows from Lemma~\ref{pointdecomp} applied to the polytope $\cQ_O(a_1,\dots,a_n/2)$ and the point $y$ with $y_{i,-i}=x_{i,-i}/2$ and other coordinates the same as in $x$. The corresponding set $D$ in this case is empty, of course.

Now let $D$ be the set of those $i$ for which $x_{i,-i}$ is odd. Consider the point $x'$ with $x'_{i,-i}=x_{i,-i}+1$ for $i\in D$ and other coordinates the same as in $x$. One easily checks that $x'$ also satisfies Definition~\ref{polytopedefB} and hence lies in $\cQ^\rB_O(\la)$. Indeed, the first two conditions are obvious while in the third condition one has two possibilities. Either all appearing coordinates are the same for $x$ and $x'$ or $(r,s)=(i,-i)$ for some $i\in D$. In the latter case in the inequality for $x$ the left-hand side is in $\bZ$ while the right-hand side is in $1/2+\bZ$. This means that the inequality for $x$ is strict, hence it also holds for $x'$.  

Since $x'_{i,-i}$ is even for all $i$, we have a decomposition $x'=x^{J_1,\varnothing}+\dots+x^{J_m,\varnothing}$ of the desired form. Note that $J_m$ contains all $(i,j)$ with $x_{i,j}\neq 0$ (which is equivalent to $x'_{i,j}\neq0$). This lets us consider the point $x^{J_m,D}$, we obtain \[x=x^{J_1,\varnothing}+\dots+x^{J_{m-1},\varnothing}+x^{J_m,D}.\]
Uniqueness follows from the fact that $J_m$ must be the smallest order ideal containing all $(i,j)$ with $x_{i,j}\neq 0$ and $D$ must be as above, $J_{m-1}$ must be the smallest order ideal containing all $(i,j)$ with $x_{i,j}-x^{J_m,D}_{i,j}\neq 0$ and so on.
\end{proof}

\begin{cor}
\hfill
\begin{enumerate}[label=(\alph*)]
\item For $k\le n-1$ the set of lattice points in $\cQ^\rB_O(\om_k)$ consists of all points of the form $x^{J,D}$ with $J\in\cJ_k$.
\item The set of lattice points in $\cQ^\rB_O(2\om_n)$ consists of all points of the form $x^{J,D}$ with $J\in\cJ_n$.
\end{enumerate}    
\end{cor}

\begin{remark}
Polytopes obtained from order polytopes by scaling along coordinates are studied in~\cite{BS} under the name \textit{lecture hall polytopes} and exhibit interesting lattice point properties. It seems plausible that this work can be extended to a theory of (lecture hall) MCOPs by generalizing the above results to general posets.
\end{remark}



\subsection{Transformed type B poset polytopes}

We now define a transformed version of $\cQ^\rB_O(\la)$ similarly to the type C case, however, here the transformation also includes a projection onto $\bR^{P\bs A}$. Consider the linear map $\xi:\bR^P\to\bR^P$ as in Definition~\ref{defxi} (where the values $r(i,j)$ are given by Definition~\ref{twisted}). Let $\pi$ denote the standard projection $\bR^P\to \bR^{P/A}$. We set $\Pi^\rB_O(\la)=\pi\xi(\cQ^\rB_O(\la))$. 

Note that for any $x\in\bR^P$ and $i\in[1,n]$ one has $\sum_j\xi(x)_{i,j}=x_{i,i}$. Let $W_\la\subset\bR^P$ be the affine subspace of points $x$ with $x_{i,i}=\la(i)$. One sees that $\xi(W_\la)$ consists of $x$ such that $\sum_j x_{i,j}=\la(i)$ for every $i$. Hence, $\pi\xi$ maps $W_\la$ bijectively onto $\bR^{P\bs A}$. In particular, $\pi\xi$ is an affine equivalence between $\cQ^\rB_O(\la)$ and $\Pi^\rB_O(\la)$. Furthermore, we have the following.
\begin{proposition}\label{latticesmap}
Consider integral dominant $\la=(a_1,\dots,a_n)$. 
\begin{enumerate}[label=(\alph*)]
\item If $a_n$ is even, then the affine lattice $W_\la\cap \bZ^P$ is mapped bijectively to $\bZ^{P\bs A}$ by $\pi\xi$.
\item If $a_n$ is odd, then the affine lattice $W_\la\cap(\one_O/2+\bZ^P)$ is mapped bijectively to $\bZ^{P\bs A}$ by $\pi\xi$.
\end{enumerate}
\end{proposition}
\begin{proof}
For (a) first note that $\pi\xi$ maps $W_\la\cap \bZ^P$ injectively into $\bZ^{P\bs A}$. Now, for $y\in\bZ^{P\bs A}$ we have $\pi\xi(x)=y$ where $x=\xi^{-1}(x')$ and $x'\in\bZ^P$ is uniquely determined by $x'_{i,j}=y_{i,j}$ for $(i,j)\in P\bs A$ and $x'\in \xi(W_\la)$ (recall that $\xi$ is unimodular).

For (b) note that \[(W_\la\cap(\one_O/2+\bZ^P))-\one_O/2=W_{\la-\om_n}\cap\bZ^P\] and the above is mapped bijectively onto $\bZ^{P\bs A}$. However, Definition~\ref{defxi} implies that $\xi(\one_O)=\sum_i\epsilon_{i,r(i,j_i)}$ where $j_i$ is $\lessdot$-maximal such that $(i,j_i)\in O$. Since the $i$th pipes of $O$ and $M_O(\al i,j_i\ar)$ coincide, we have $r(i,j_i)=i$ and ${\pi\xi(\one_O)=0}$.
\end{proof}

From here on we largely shift our attention to the polytopes $\Pi^\rB_O(\la)$ and their lattice points. The above proposition shows that those are in bijection with the set \[\cQ^\rB_O(\la)\cap(a_n\one_O/2+\bZ^P).\] We point out, however, that $\Pi^\rB_O(\la)$ will still not be a lattice polytope when $a_n$ is odd even though we will be studying its lattice points (see Example~\ref{piBex}).

\begin{proposition}\label{imagexJD}
For $x^{J,D}$ with $J\in\cJ_k$ one has $\pi\xi(x^{J,D})_{i,j}=0$ when $i>k$, otherwise
\[
\pi\xi(x^{J,D})_{i,j}=
\begin{cases}
0&\text{ if }j\neq w^{O,J}(i),\\
1&\text{ if }j=w^{O,J}(i)\text{ and }((i,j)\notin B\text{ or }i\in D),\\
2&\text{ if }j=w^{O,J}(i)\text{ and }(i,j)\in B\text{ and }i\notin D.
\end{cases}
\] 
\end{proposition}
\begin{proof}
The first claim is immediate since $x^{J,D}_{i,j}=0$ whenever $i>k$.

Recall that $w^{O,J}(i)=r(i,j)$ for the $\lessdot$-maximal $j$ such that $(i,j)\in M_O(J)$. Since $(i,-i)\notin O$, we have $r(i,-i)=-i$ for all $i$. Hence, $w^{O,J}(i)=-i$ holds if and only if $(i,-i)\in J$. Let $D'$ consist of all $i$ for which $(i,-i)\in J$, we see that the last case in the proposition's statement occurs if and only if $j=-i$ and $i\in D'\bs D$. 

Thus, if $D=D'$, the claim follows directly from Proposition~\ref{pointimage} because in this case $x^{J,D}=\one_{M_O(J)}$. To pass to the general case note that \[x^{J,D}=x^{J,D'}+\sum_{i\in D'\bs D}\epsilon_{i,-i}.\] Hence, we are to check that $\pi\xi(\epsilon_{i,-i})=\epsilon_{i,-i}$. Note that we have $r(i,j)=i$ for the $\lessdot$-maximal $j$ such that $(i,j)\in O$. We obtain \[\pi\xi(\epsilon_{i,-i})=\pi(\epsilon_{i,r(i,-i)}-\epsilon_{i,r(i,j)})=\epsilon_{i,-i}.\qedhere\]
\end{proof}

We now give a description of the set $\Pi^\rB_O(\la)\cap\bZ^{P\bs A}$. For $D\subset[1,n]$ let $y^D\in\bR^{P\bs A}$ be the point with $y^D_{i,-i}=1$ for $i\in D$ and all other coordinates zero.

\begin{lemma}\label{pipointdecompB}
For an integral dominant $\la=(a_1,\dots,a_n)$ consider $y\in\Pi^\rB_O(\la)\cap\bZ^{P\bs A}$.
\begin{enumerate}[label=(\alph*)]
\item If $a_n$ is even, there exist  unique sequence of order ideals $J_1\subset\cdots\subset J_m$ in $\cJ$ and set $D\subset[1,n]$ such that \[y=\pi\xi(x^{J_1,\varnothing})+\dots+\pi\xi(x^{J_{m-1},\varnothing})+\pi\xi(x^{J_m,D})\] and for $k\le n-1$ exactly $a_k$ of the $J_i$ lie in $\cJ_k$ while $a_n/2$ of the $J_i$ lie in $\cJ_n$.  
\item If $a_n$ is odd, there exist unique sequence of order ideals $J_1\subset\cdots\subset J_m$ in $\cJ$ and set $D\subset[1,n]$ such that \[y=\pi\xi(x^{J_1,\varnothing})+\dots+\pi\xi(x^{J_m,\varnothing})+y^D\] and for $k\le n-1$ exactly $a_k$ of the $J_i$ lie in $\cJ_k$ while $(a_n-1)/2$ of the $J_i$ lie in $\cJ_n$.    
\end{enumerate}
\end{lemma}
\begin{proof}
Part (a) is immediate from Proposition~\ref{latticesmap} and Lemma~\ref{pointdecompB}.

For part (b) let $D$ be the set of all $i$ for which $y_{i,-i}$ is odd. By part (a) it suffices to show that $y-y^D\in\Pi^\rB_O(\la-\om_n)$. We have a unique point $x\in W_\la\cap(\one_O/2+\bZ^P)$ with $\pi\xi(x)=y$. However, we have seen that $\pi\xi(\epsilon_{i,-i})=\epsilon_{i,-i}$ and $\pi\xi(\one_O)=0$, hence $\pi\xi(x^D)=y^D$ where $x^D=\one_{\{(i,-i)|i\in D\}}+\one_O/2$. Therefore, it suffices to check that $x-x^D\in\cQ^\rB_O(\la-\om_n)$. 

To do so one verifies that $x-x^D$ satisfies Definition~\ref{polytopedefB}. The first two conditions are immediate from the corresponding conditions for the point $x$ and polytope $\cQ^\rB_O(\la)$. In the third condition, when passing from $x$ to $x-x^D$ the left-hand side does not change. The right-hand changes only if $(r,s)\notin O$ and $(r,s)\neq (i,-i)$ for any $i\in D$. In the latter case, however, the right-hand side decreases by $1/2$ originally being in $1/2+\bZ$ and becoming an integer. Since the right-hand side is an integer, the inequality remains true. 
\end{proof}

\begin{cor}\label{spinpoints}
The set $\Pi^\rB_O(\om_n)\cap\bZ^{P\bs A}$ consists of all points of the form $y^D$.
\end{cor} 

\begin{cor}\label{pointsdim}
For any integral dominant $\la$ one has $|\Pi^\rB_O(\la)\cap\bZ^{P\bs A}|=\dim V_\la$.    
\end{cor}
\begin{proof}
Lemma~\ref{pipointdecompB} shows that the number of integer points in $\Pi^\rB_O(\la)$ does not depend on $O$. However, in the case of $O=P\bs B$ the sets $\cQ^\rB_O(\la)\cap\bZ^P$ when $\la(1)\in\bZ$ and $\cQ^\rB_O(\la)\cap(\one_O/2+\bZ^P)$ when $\la(1)\in 1/2+\bZ$ are studied in~\cite{BZ} and shown to have size $\dim V_\la$. See definition of an $o_{2r+1}$-pattern and Theorem 4.2 in loc.\ cit. By Proposition~\ref{latticesmap} these sets are in bijection with $\Pi^\rB_O(\la)\cap\bZ^{P\bs A}$. 

Alternatively, in the case of $O=A$ one has $\xi(x)_{i,j}=x_{i,j}$ for $(i,j)\in P\bs A$ and hence $\Pi^\rB_A(\la)=\pi(\cQ^\rB_A(\la))$. Such a polytope is shown to have $\dim V_\la$ integer points in~\cite[Corollary 2.2]{MtypeB}.
\end{proof}

\begin{cor}\label{minkowskiB}
For integral dominant $\la=(a_1,\dots,a_n)$ and $\mu=(b_1,\dots,b_n)$ with at least one of $a_n$ and $b_n$ even one has \[\Pi^\rB_O(\la+\mu)\cap\bZ^{P\bs A}=\Pi^\rB_O(\la)\cap\bZ^{P\bs A}+\Pi^\rB_O(\mu)\cap\bZ^{P\bs A}.\]
\end{cor}
\begin{proof}
Since $\Pi^\rB_O(\la+\mu)=\Pi^\rB_O(\la)+\Pi^\rB_O(\mu)$, we are to show that any $x\in \Pi^\rB_O(\la+\mu)\cap\bZ^{P\bs A}$ can be decomposed into a sum of two points, one in $\Pi^\rB_O(\la)\cap\bZ^{P\bs A}$ and the other in $\Pi^\rB_O(\mu)\cap\bZ^{P\bs A}$. However, this is seen directly from Lemma~\ref{pipointdecompB} by decomposing $x$ and then splitting the summands into two groups in the appropriate way.    
\end{proof}

\begin{example}\label{piBex}
Let $n=2$, $O=\{(1,1),(2,2),(1,-2)\}$. First consider $\cQ^\rB_O(\om_1)$. Its four vertices are the points $x^{J,\varnothing}$ with $J\in\cJ_1$: $(1,0,0,0,0,0)$, $(1,1,0,0,0,0)$, $(1,0,1,0,0,0)$, $(1,0,1,2,0,0)$ where $x\in\bR^P$ is written as $(x_{1,1},x_{1,2},x_{1,-2},x_{1,-1},x_{2,2},x_{2,-2})$. It will also have one more integer point: $x^{J,\{1\}}=(1,0,1,1,0,0)$ where $J=\{(1,j)\}_{j\in N}$. Meanwhile, $\cQ^\rB_O(\om_2)$ has five vertices: $(\frac12,0,0,0,\frac12,0)$, $(\frac12,0,\frac12,0,\frac12,0)$, $(\frac12,0,\frac12,1,\frac12,0)$, $(\frac12,0,\frac12,0,\frac12,1)$, $(\frac12,0,\frac12,1,\frac12,1)$; these are the $x^{J,\varnothing}/2$ with $J\in\cJ_2$. The last four form $\cQ^\rB_O(\om_2)\cap(\one_O/2+\bZ^P)$.

Next one checks that $r(1,1)=-2$, $r(1,2)=2$, $r(1,-2)=1$, $r(1,-1)=-1$, $r(2,2)=2$, $r(2,-2)=-2$ and, writing $y\in\bR^{P\bs A}$ as $(y_{1,2},y_{1,-2},y_{1,-1},y_{2,-2})$, derives 
\[\pi \xi(x_{1,1},x_{1,2},x_{1,-2},x_{1,-1},x_{2,2},x_{2,-2})=
(x_{1,2},x_{1,1}-x_{1,2}-x_{1,-2},x_{1,-1},x_{2,-2}).\]
One sees that $\Pi^\rB_O(\om_1)$ has vertices $(0,1,0,0)$, $(1,0,0,0)$, $(0,0,0,0)$, $(0,0,2,0)$ and one more integer point $(0,0,1,0)$. This matches $\dim V_{\om_1}=5$ (Corollary~\ref{pointsdim}). Meanwhile, $\Pi^\rB_O(\om_2)$ will have five vertices:
$(0,\frac12,0,0)$, $(0,0,0,0)$, $(0,0,1,0)$, $(0,0,0,1)$, $(0,0,1,1)$
and four integer points (the last four). This matches $\dim V_{\om_2}=4$. 
One may check that $\Pi^\rB_O(2\om_2)=2\Pi^\rB_O(\om_2)$ has $\dim V_{2\om_2}=10$ integer points and $\Pi^\rB_O(\om_1+\om_2)=\Pi^\rB_O(\om_1)+\Pi^\rB_O(\om_2)$ has $\dim V_{\om_1+\om_2}=16$ integer points (in the latter case every integer point is a sum of integer points in the summands in view of Corollary~\ref{minkowskiB}, in the former case this is not true).
\end{example}

\begin{remark}
Evidently, versions of Corollaries~\ref{pointsdim} and~\ref{minkowskiB} can also be stated for the type B poset polytopes $\cQ^\rB_O(\la)$ where instead of integer point sets one would consider intersections with the respective lattices from Proposition~\ref{latticesmap}.
\end{remark}

\subsection{Essential signatures}

In type B we first construct PBW-monomial bases while the existence of other objects is derived therefrom. We obtain these bases using the language of essential signatures due to Vinberg.

For $d\in\bZ_{\ge0}^{P\bs A}$ let $f^d$ denote the PBW monomial $\prod_{i,j}f_{i,j}^{d_{i,j}}\in\cU(\fg)$ ordered first by $i$ increasing from left to right and then by $j$ increasing with respect to $\lessdot$ from left to right. 

\begin{definition}
For a total group order $<$ on $\bZ^{P\bs A}$ and an integral dominant weight $\la$ we say that a vector (or \textit{signature}) $d\in \bZ_{\ge0}^{P\bs A}$ is \textit{essential} for $<$ and $\la$ if \[f^d v_\la\notin\spann\{f^{d'} v_\la|d'<d\}.\]
A PBW monomial $f^d$ is essential for $<$ and $\la$ if $d$ is.
\end{definition}

Evidently, for any total group order $<$ the set of vectors $f^dv_\la$ such that $d$ is essential for $<$ and $\la$ forms a basis in $V_\la$. The following is a key property of essential signatures, we give a proof since sources such as~\cite{Gor15,FFL3} place restrictions on the group order.
\begin{proposition}\label{semigroup}
If $d$ is essential for $<$ and $\la$ while $d'$ is essential for $<$ and $\mu$, then $d+d'$ is essential for $<$ and $\la+\mu$.
\end{proposition}
\begin{proof}
For $c\in\bZ_{\ge0}^{P\bs A}$ and integral dominant $\nu$ let $V_\nu[{<}c]\subset V_\nu$ denote the subspace spanned by all vectors $f^{c'}v_\nu$ with $c'<c$. We are to show that $f^{d+d'}v_{\la+\mu}\notin V_{\la+\mu}[{<}d+d']$. 

Let us view $V_{\la+\mu}$ as a submodule of $V_\la\otimes V_\mu$ so that $v_{\la+\mu}$ coincides with $v_\la\otimes v_\mu$. Then \[V_{\la+\mu}[{<}d+d']\subset U_{d,d'}= V_\la[{<}d]\otimes V_\mu+V_\la\otimes V_\mu[{<}d'].\] Indeed, for $c<d+d'$ we can expand $f^c(v_\la\otimes v_\mu)$ into a sum of vectors of the form $(f^{c_1}v_\la)\otimes(f^{c_2}v_\mu)$ with $c_1+c_2=c$ and, therefore, $c_1<d$ or $c_2<d'$. 

Now, if we expand $f^{d+d'}(v_\la\otimes v_\mu)$ in similar fashion, then we will obtain a sum of $(f^dv_\la)\otimes(f^{d'}v_\mu)$ and other products $(f^{c_1}v_\la)\otimes(f^{c_2}v_\mu)$ with $c_1<d$ or $c_2<d'$. Thus, \[f^{d+d'}v_{\la+\mu}\in(f^dv_\la)\otimes(f^{d'}v_\mu)+U_{d,d'}\] and it remains to show that $(f^dv_\la)\otimes(f^{d'}v_\mu)\notin U_{d,d'}$. However, this follows from $f^dv_\la\notin V_\la[{<}d]$ and $f^{d'}v_\mu\notin V_\mu[{<}d']$, i.e. the essentiality of $d$ and $d'$.
\end{proof}

\subsection{PBW-monomial bases}

We now fix a total group order $<$ and simply write ``essential for $\la$'' presuming $<$. For $d\in\bZ^{P\backslash A}$ set $d(i)=\sum_{|j|>i}d_{i,j}+d_{i,-i}/2$.
\begin{definition}\label{monorder}
First let us define a total order on $P\bs A$. Every $(i,j)\in P\bs A$ can be uniquely expressed as $(i,r(i,j'))$. We set $(i_1,r(i_1,j_1))<(i_2,r(i_2,j_2))$ if $i_1<i_2$ or ($i_1=i_2$ and $j_1\gtrdot j_2$) with one exception.
If $i_1=i_2$, $j_1\gtrdot j_2$, $(i_2,j_2)\in O$ and there is no $j_2\lessdot j'\ledot j_1$ with $(i_1,j')\in O$, we set $(i_1,r(i_1,j_1))>(i_2,r(i_2,j_2))$.

Now consider distinct $d,d'\in\bZ^{P\bs A}$. Consider the $<$-minimal $(i,j)$ such that $d_{i,j}\neq d'_{i,j}$. If $d'(i)<d(i)$, we set $d'<d$. If $d'(i)=d(i)$, we set $d'<d$ when $d'_{i,j}>d_{i,j}$.
\end{definition}

In other words, the order $<$ on elements $(i,j)$ with a fixed $i$ is reverse to the one in Definition~\ref{monorderC}. Note, however, that for the $\lessdot$-maximal $l$ with $(i,l)\in O$ we have $r(i,l)=i$, hence $(i,r(i,l))\notin P\bs A$ and is excluded from the order instead of being $<$-minimal among the $(i,j)$. Instead, the $<$-minimal $(i,j)$ is $(i,r(i,-i))=(i,-i)$.

The main result of this subsection is as follows.

\begin{theorem}\label{essentialB}
For integral dominant $\la$ a signature is essential for $\la$ if and only if it lies in $\Pi^\rB_O(\la)$. Consequently, the vectors $f^yv_\la$ with $y\in\Pi^\rB_O(\la)\cap\bZ^{P\bs A}$ form a basis in $V_\la$.
\end{theorem}

Theorem~\ref{essentialB} is proved by induction on $\la$ using Proposition~\ref{semigroup} and Corollary~\ref{minkowskiB}. The most difficult part of the proof is the induction base, i.e.\ the case of fundamental weights and $2\om_n$. We next give a series of preparatory definitions and statements. We make use of the isomorphism $V_{\om_k}\simeq\wedge^k V$ and consider the actions of essential monomials in terms of multivectors. 

In a PBW monomial $f^d$ we may express every root vector as a sum of two $E_{a,b}$ according to~\eqref{matrixform} and expand this product into a sum of products of the $E_{a,b}$ which can be viewed as elements of $\cU(\fgl_{2n+1}(\bC))$. If a product $E_{a_1,b_1}\dots E_{a_m,b_m}$ occurs in this expansion we say that it \textit{occurs} in $f^d$ (note that the summands in the expansion are pairwise distinct). We also say that $e_{i_1,\dots,i_k}$ \textit{occurs} in $v\in\wedge^k V$ if it appears with a nonzero coefficient in the expression of $v$ in the multivector basis.

If $\{i_1,\dots,i_k\}$ contains $b$ but not $a$, the matrix $E_{a,b}$ maps $e_{i_1,\dots,i_k}$ to $e_{j_1,\dots,j_k}$ where $j_1,\dots,j_k$ is obtained from $i_1,\dots,i_k$ by replacing $b$ with $a$. Otherwise $E_{a,b}e_{i_1,\dots,i_k}=0$. Hence, every product of the $E_{a,b}$ maps the highest weight vector $e_{1,\dots,k}$ to another multivector or to 0. In particular, if $e_{i_1,\dots,i_k}$ occurs in $f^d e_{1,\dots,k}$, then we must have a product $C=E_{a_1,b_1}\dots E_{a_m,b_m}$ such that $C$ occurs in $f^d$ and $Ce_{1,\dots,k}=\pm e_{i_1,\dots,i_k}$.


To prove that $\Pi^\rB_O(\om_k)\cap\bZ^{P\bs A}$ is the set of essential signatures for $\om_k$ we will use the following approach. For every $d$ in this set (that is $d=\pi\xi(x^{J,D})$ for some $J\in\cJ_k$) we choose a distinguished product $C_{J,D}$ occurring in $f^d$ in such a way that $C_{J,D}e_{1,\dots,k}$ occurs in $f^de_{1,\dots,k}$ but not in any $f^{d'}e_{1,\dots,k}$ with $d'<d$.

\begin{definition}\label{ejddef}
For $k\in[1,n]$ and a point $x^{J,D}$ with $J\in\cJ_k$ consider $i\in[1,k]$ and $j=w^{O,J}(i)$. Let $r$ denote the number of $i'>i$ with $i'\in D$. Set 
\[
E^{J,D}_i=
\begin{cases}
E_{j,i}&\text{ if }|j|>i\text{ and }r\text{ is even,}\\
E_{-i,-j}&\text{ if }|j|>i\text{ and }r\text{ is odd,}\\
E_{0,i}&\text{ if }i=-j,\, i\in D\text{ and }r\text{ is even,}\\
E_{-i,0}&\text{ if }i=-j,\, i\in D\text{ and }r\text{ is odd,}\\
E_{-i,0}E_{0,i}&\text{ if }i=-j,\, i\notin D\text{ and }r\text{ is even,}\\
E_{0,i}E_{-i,0}&\text{ if }i=-j,\, i\notin D\text{ and }r\text{ is odd,}\\
1&\text{ if }i=j.
\end{cases}
\]
Let $C_{J,D}\in\cU(\fgl_{2n+1}(\bC))$ denote the PBW monomial $\prod_{i=1}^k E^{J,D}_i$ ordered by $i$ increasing from left to right. Let $e_{J,D}$ denote the image $C_{J,D}e_{1,\dots,k}$.
\end{definition}

\begin{example}\label{ejdexample}
Consider the case $n=4, O=\{(1,1),(1,-3),(2,2),(2,-3),(3,3),(4,4)\}$. Consider $J=\langle 1,-1\ar\cup\al3,-3\rangle\in\mathcal{J}_3$. We have
\[w^{O,J}(1)=-1,w^{O,J}(2)=-3,w^{O,J}(3)=3,w^{O,J}(4)=-4.\]
Hence, $C_{J,D}$ and $e_{J,D}$ can be written out as below. Here note that the $e_{J,D}$ are pairwise distinct, this actually holds for all possible $e_{J,D}$ (Corollary~\ref{bijection}). Also, the reader may compare this example to the general description of $e_{J,D}$ given by Proposition~\ref{ejd}.
\begin{align*}
C_{J,\varnothing}=E_{-1,0}E_{0,1}E_{-3,2}E_{-4,0}E_{0,4},&~e_{J,\varnothing}=e_{-1,-3,3,-4};\\
C_{J,\{1\}}=E_{0,1}E_{-3,2}E_{-4,0}E_{0,4},&~e_{J,\{1\}}=e_{0,-3,3,-4};\\
C_{J,\{4\}}=E_{0,1}E_{-1,0}E_{-2,3}E_{0,4},&~e_{J,\{3\}}=e_{0,2,-2,-1};\\
C_{J,\{1,4\}}=E_{-1,0}E_{-2,3}E_{0,4},&~e_{J,\{1,3\}}=e_{1,2,-2,-1}.
\end{align*}
\end{example}

First of all, let us point out that $C_{J,D}$ indeed occurs in $f^d$ where $d={\pi\xi(x^{J,D})}$, since it is obtained by choosing one of the two summands in~\eqref{matrixform} for every $f_{i,j}$ in the product. Now, by Proposition~\ref{imagexJD} every $f_{i,j}$ with $|j|>i$ occurs in $f^d$ in degree 0 or 1 while $f_{i,-i}$ can occur in degrees 0, 1 and 2. Also note that at most one of $E_{0,i}$ and $E_{-i,0}$ acts nontrivially on any multivector. This allows us to describe $C_{J,D}$ in the following way. For a root vector $f_{i,j}$ with $|j|>i$ appearing in $f^d$ we choose the first summand in~\eqref{matrixform} if the number of root vectors of the form $f_{i',-i'}$ appearing in $f^d$ in degree 1 and to the right of $f_{i,j}$ is even. We choose the second summand if this number is odd. For a root vector $f_{i,-i}$ appearing in $f^d$ we choose the only summand which acts nontrivially considering the choices we have made for factors to the right. The latter is possible and $e_{J,D}$ is indeed a nonzero multivector in view of Proposition~\ref{ejd} which describes $e_{J,D}$ explicitly. First, however, we prove a lemma.

\begin{lemma}\label{between}
Consider $J\in\cJ_k$, let $D$ denote the set of $i$ for which $(i,-i)\in J$. For $i\in D$ one has $w^{O,J}(i)=-i$. Furthermore, for $i_1,i_2\in D\cup\{0,n+1\}$ and $i_1<i<i_2$ one has $i_1<|w^{O,J}(i)|<i_2$.
\end{lemma}
\begin{proof}
We have seen that $r(i,-i)=-i$. For $i\in D$ we have $(i,-i)\in M_O(J)$ and $w^{O,J}(i)=r(i,-i)$ by Proposition~\ref{wrproperties}\ref{wOisr}, this implies the first claim.

For the second claim recall that $w^{O,J}(i)=r(i,j)$ for the $\lessdot$-maximal $j$ such that $(i,j)\in M_O(J)$. We get $|r(i,j)|\ge i>i_1$ from Proposition~\ref{wrproperties}\ref{rbound1}. 

Now first suppose that $(i,j)\notin O$. Then $(i,j)\in\max_\prec J$ and, since $(i_1,-i_1),(i_2,-i_2)\in\max_\prec J$, we must have $-i_2\lessdot j\lessdot -i_1$. We deduce $|r(i,j)|\le -j<i_2$ from Proposition~\ref{wrproperties}\ref{rbound1}. If, however, $(i,j)\in O$, then we have two possibilities. First: there exists a $j'\gtrdot j$ with $(i,j')\in O$ which implies $(i,j')\notin J$. Hence $-i_2\lessdot j'\lessdot -i_1$ and Proposition~\ref{wrproperties}\ref{rbound2} provides $|r(i,j)|\le -j'<i_2$. Second: no such $j'$ exists and $r(i,j)=i$.
\end{proof}

In other words, the above proposition shows that for $i_1,i_2\in D\cup\{0,n+1\}$ with $i_1<i_2$ the image $w^{O,J}(]i_1,i_2[)$ consists of numbers whose absolute value is again in $]i_1,i_2[$ (we write $]i_1,i_2[$ for the set of all integers $i_1<i<i_2$ and similarly use the notations $[i_1,i_2[$ and $]i_1,i_2]$). In particular, when $i_2\neq n+1$ precisely half of the numbers with absolute value in $]i_1,i_2[$ appear in the image $w^{O,J}([1,k])$.

\begin{proposition}\label{ejd}
Consider $x^{J,D}$ with $J\in\cJ_k$ and $D=\{i_1>\dots>i_l\}$, additionally denote $i_{l+1}=0$. Then the vector $e_{J,D}$ is a nonzero multivector. Furthermore, let $K\subset[-n,n]$ with $|K|=k$ be the set of subscripts of this multivector. Then $j\in K$ if and only if one of the following holds.
\begin{enumerate}[label=(\roman*)]
\item $j\in w^{O,J}(]i_1,k])$ or
\item $j\in w^{O,J}(]i_{2r+1},i_{2r}[)$ for some $r\in[1,\lfloor \frac l2\rfloor]$ or
\item $|j|\in]i_{2r},i_{2r-1}[$ for some $r\in[1,\lfloor \frac{l+1}2\rfloor]$ and $j\neq -w^{O,J}(i)$ for any $i\in]i_{2r+1},i_{2r}[$ or
\item $j=\pm i_{2r}$ for some $r\in[1,\lfloor \frac l2\rfloor]$ or
\item $j=0$ and $l$ is odd.
\end{enumerate}
\end{proposition}
\begin{proof}
By Lemma~\ref{between} we have $|w^{O,J}(i)|>i_1$ if and only if $i>i_1$. By the same proposition for $s\in[0,l-1]$ we have $|w^{O,J}(i)|\in ]i_{s+1},i_s[$ if and only if $i\in ]i_{s+1},i_s[$. Hence, the number of $j$ satisfying condition (i) is $k-i_1$ and the number of $j$ satisfying condition (ii) for a given $r$ is $i_{2r}-i_{2r+1}-1$ while for condition (ii) it is $i_{2r-1}-i_{2r}-1$. We see that the total number of $j$ satisfying (i), (ii) or (iii) is $k-l$. Since conditions (iv) and (v) provide $l$ further numbers, we deduce that in total precisely $k$ distinct $j\in[-n,n]$ satisfy one of the conditions and the first claim follows from the second.

Now for $s\in [1,k]$ denote \[e_{J,D}^s=E^{J,D}_s\dots E^{J,D}_k e_{1,\dots,k}\] so that $e_{J,D}^1=e_{J,D}$. Let $K_s$ denote the set of subscripts of the multivector $e_{J,D}^s$. Note that for $i\in]i_1,k]$ the element $E^{J,D}_i$ is equal to $E_{w^{O,J}(i),i}$, $E_{-i,0}E_{0,i}$ or $1$. In all three cases $E^{J,D}_i$ acts by replacing the subscript $i$ with $w^{O,J}(i)$. Since $|w^{O,J}(i)|\ge i$ (Proposition~\ref{wrproperties}\ref{rijgei}), one sees that every $e_{J,D}^i$ with $i\in]i_1,k]$ is nonzero and $K_{i_1+1}$ is obtained from $[1,k]$ by replacing the subset $]i_1,k]$ with the image $w^{O,J}(]i_1,k])$. The set $K_{i_1}$ is obtained from $K_{i_1+1}$ by replacing $i_1$ with 0 since $E^{J,D}_{i_1}=E_{0,i_1}$. 

Next, for $i\in ]i_2,i_1[$ the element $E^{J,D}_i$ equals $E_{-i,-w^{O,J}(i)}$, $E_{0,i}E_{-i,0}$ or $1$ but in all three cases it replaces $-w^{O,J}(i)$ with $-i$ if possible. We claim that this is indeed always possible and hence each of the corresponding $e_{J,D}^i$ is nonzero. Suppose that, on the contrary, $i'\in]i_2,i_1[$ is the largest number for which $e_{J,D}^{i'}=0$. Since for $i\in]i',i_1[$ we have $|w^{O,J}(i)|\in[i,i_1[$, we see that $K_{i'+1}$ is the union of $[1,i']$, the image $w^{O,J}(]i_1,k])$, the singleton $\{0\}$ and the difference 
\[\big({[i'+1,i_1[}\cup{]-i_1,-i'-1]}\big)\bs \{-w^{O,J}(i)\}_{i\in[i'+1,i_1[}.\] 
In particular, $j$ with $|j|\in[i'+1,i_1[$ is not contained in $K_{i'+1}$ if and only $j=-w^{O,J}(i)$ for some $i\in[i'+1,i_1[$. Since $-w^{O,J}(i')$ is clearly not of this form, the remaining possibility is $|w^{O,J}(i)|=i$ but both $E_{0,i'}E_{-i',0}$ and $1$ act nontrivially on $e_{J,D}^{i'+1}$. This allows us to describe $K_{i_2+1}$ and $K_{i_2}$ differs from $K_{i_2+1}$ by replacing $0$ with $-i_2$.

By similar arguments we see that, more generally, 
\begin{itemize}
\item $K_{i_{2r+1}+1}$ differs from $K_{i_{2r}}$ by replacing $]i_{2r+1},i_{2r}[$ with $w^{O,J}(]i_{2r+1},i_{2r}[)$, 
\item $K_{i_{2r+1}}$ differs from $K_{i_{2r+1}+1}$ by replacing $i_{2r+1}$ with 0,
\item $K_{i_{2r+2}+1}$ differs from $K_{i_{2r+1}}$ by replacing $]i_{2r+2},i_{2r+1}[$ with 
\[\big({]i_{2r+2},i_{2r+1}[}\cup{]-i_{2r+1},-i_{2r+2}[}\big)\bs \{-w^{O,J}(i)\}_{i\in]i_{2r+2},i_{2r+1}[},\]
\item $K_{i_{2r+2}}$ differs from $K_{i_{2r+2}+1}$ by replacing 0 with $-i_{2r+2}$.
\end{itemize}
The proposition follows.
\end{proof}

Our choice of $e_{J,D}$ attaches a nonzero multivector occurring in $f^{\pi\xi(x^{J,D})} e_{1,\dots,k}$ to every lattice point $\pi\xi(x^{J,D})$ with $J\in\cJ_k$. We claim that this is a bijection between the set of lattice points in $\Pi^\rB_O(\om_k)$ (or $\Pi^\rB_O(2\om_n)$ if $k=n$) and the basis of multivectors in $\wedge^k V$. Next, one may order the points $\pi\xi(x^{J,D})$, $J\in\cJ_k$ according to $<$ which induces an ordering of the vectors $f^{\pi\xi(x^{J,D})}e_{1,\dots,k}$ and also of the multivectors $e_{J,D}$. These two orderings define a square matrix expressing the first set of vectors via the second set of vectors (the second set is a basis by the previous claim). We further claim that this matrix is triangular and nondegenerate which would then imply that the $f^{\pi\xi(x^{J,D})}e_{1,\dots,k}$ also form a basis in $\wedge^k V$. These claims are immediate from the following.
\begin{proposition}\label{baseB}
For a point $x^{J,D}$ with $J\in\cJ_k$ denote $d=\pi\xi(x^{J,D})$. Then the multivector $e_{J,D}$ occurs in $f^d e_{1,\dots,k}$ and does not occur in $f^{d'} e_{1,\dots,k}$ for any $d'<d$.
\end{proposition}
\begin{proof}
Suppose that $e_{J,D}$ occurs in $f^{d'} e_{1,\dots,k}$ for some $d'\le d$. Then we have a product $C'$ occurring in $f^{d'}$ such that $C'e_{1,\dots,k}=\pm e_{J,D}$. We prove that $C'=C_{J,D}$, this will imply that $d'=d$ (because a summand $C'$ occurring in $f^{d'}$ determines $d'$). It will also imply that  $C''e_{1,\dots,k}\neq\pm e_{J,D}$ for any other $C''$ occurring in $f^d$ and hence $e_{J,D}$ occurs in $f^d e_{1,\dots,k}$.

Let the set $\mathcal E_i$ consist of all $E_{j,i}$ and $E_{-i,j}$ for which $|j|>i$ or $j=0$, i.e.\ all those $E_{a,b}$ which occur in some $f_{i,j}$. Denote $C=C_{J,D}$, both $C$ and $C'$ when read from left to right start with a (possibly empty) product of elements of $\mathcal E_1$, followed by a product of elements of $\mathcal E_2$, etc. Since $C$ and $C'$ act nontrivially on $e_{1,\dots,k}$, the last $n-k$ products are empty in both cases.

For $i\in[1,n+1]$ write $C[{<}i]$ (resp.\ $C'[{<}i]$) to denote the subword in $C$ (resp.\  $C'$) consisting of all the appearing $E_{a,b}\in\bigcup_{i'<i}\mathcal E_{i'}$ taken in the same order and to the same powers. Similarly, for $i\in[0,n]$ let $C[{>}i]$ (resp.\ $C'[{>}i]$) denote the subword in $C$ (resp.\  $C'$) consisting of all $E_{a,b}\in\bigcup_{i'>i}\mathcal E_{i'}$. By induction on $i$ we show that $C[{<}i]=C'[{<}i]$ for all $i\in[1,k+1]$. 

The induction base $i=1$ is trivial. Consider the induction step from $i\in[1,k]$ to $i+1$. In view of the induction hypothesis, $d'(i)\le d(i)$ (cf.\ Definition~\ref{monorder}). Set $j=w^{O,J}(i)$. Note that if $j=i$, then $d(i)=0$, hence $d'(i)=0$, and the induction step is immediate. 

Let $j\gtrdot i$. Suppose $d'(i)=0$, i.e.\ $C'[{>}(i-1)]=C'[{>}i]$. Both of the nonzero multivectors $C[{>}i]e_{1,\dots,k}$ and $C'[{>}i]e_{1,\dots,k}$ have $i$ as a subscript but not $-i$ because all $E_{a,b}$ appearing in $C[{>}i]$ and $C'[{>}i]$ have $|a|,|b|\neq i$. Thus, the subscript sets of the multivectors
\[C[{>}(i-1)]e_{1,\dots,k}=E^{J,D}_iC[{>}i]e_{1,\dots,k},\quad C'[{>}(i-1)]e_{1,\dots,k}=C'[{>}i]e_{1,\dots,k}\] 
are distinct because the action of $E^{J,D}_i$ removes $i$ from or adds $-i$ to the subscript set. However, both of these multivectors are mapped to $\pm e_{J,D}$ by $C[{<}i]=C'[{<}i]$. This is impossible since any product of the $E_{a,b}$ which does not vanish on two multivectors with distinct subscript sets maps them to multivectors with distinct subscript sets. 

Thus, $d'(i)>0$. Since $d(i)\in\{1,1/2\}$, there is exactly one $j'$ such that $d'_{i,j'}>0$. If $d'(i)=d(i)=1$, then $d'\le d$ implies $(i,j')\le(i,j)$ (the order in Definition~\ref{monorder}). Otherwise we have $d'(i)=1/2$ and $j'=-i$ which again implies $(i,j')\le(i,j)$ since $(i,-i)$ is $<$-minimal among all $(i,j)$.

Suppose $j=-i$, then $j'=j$. If $d_{i,-i}=d'_{i,-i}=1$ (i.e.\ $i\in D$), then exactly one of $E_{0,i}$ and $E_{-i,0}$ appears in each of $C$ and $C'$. For the induction step we need to show that these two factors are the same. Indeed, suppose that $E_{0,i}$ appears in $C$ but $E_{-i,0}$ appears in $C'$. The multivectors $E_{0,i}C[{>}i](e_{1,\dots,k})$ and $E_{-i,0}C'[{>}i](e_{1,\dots,k})$ must have the same set of subscripts since both are mapped to $\pm e_{J,D}$ by $C[{<}i]$. This is impossible since 0 is among the subscripts of $E_{0,i}C[{>}i](e_{1,\dots,k})$ but not among the subscripts of $E_{-i,0}C'[{>}i](e_{1,\dots,k})$. The case of $E_{-i,0}$ appearing in $C$ and $E_{0,i}$ appearing in $C'$ is symmetric.

Let $d_{i,-i}=2$, i.e.\ $i\notin D$. Then the subscript set of $E^{J,D}_iC[{>}i]e_{1,\dots,k}$ contains $-i$ but not $i$. If $d'_{i,-i}=1$, then $C'[{>}(i-1)]$ differs from $C'[{>}i]$ by a factor of $E_{0,i}$ or $E_{-i,0}$ and the subscript set of $C'[{>}(i-1)]e_{1,\dots,k}$ contains, respectively, neither or both $i$ and $-i$. Again we have a contradiction with 
\begin{equation}\label{samesubs}
E^{J,D}_iC[{>}i]e_{1,\dots,k}=\pm C'[{>}(i-1)]e_{1,\dots,k}.    
\end{equation} 
Hence $d'_{i,-i}=2$ and $C'[{>}(i-1)]=E'C'[{>}i]$ where $E'$ is either $E_{-i,0}E_{0,i}$ or $E_{0,i}E_{-i,0}$. We check that $E'=E^{J,D}_i$. Indeed, if, for instance, $E'=E_{-i,0}E_{0,i}$ while $E^{J,D}_i=E_{0,i}E_{-i,0}$, then $0$ is among the subscripts of $E^{J,D}_iC[{>}i]e_{1,\dots,k}$ but not $C'[{>}(i-1)]e_{1,\dots,k}$ contradicting~\eqref{samesubs}.

We may now assume that $|j|>i$, in particular, $i\notin D$. Let $C'[{>}(i-1)]=E'C'[{>}i]$, then $E'$ is a product of $d'_{i,j'}\in\{1,2\}$ factors of the form $E_{a,b}$. Denote $D=\{i_1>\dots>i_l\}$ and choose $r$ so that $i\in{}]i_{r+1},i_r[$ where $i_0=n+1$ and $i_{l+1}=0$. First, let $r$ be even meaning that $E^{J,D}_i=E_{j,i}$. The subscripts of $E_{j,i}C[{>}i]e_{1,\dots,k}$ do not include $-i$ (since $-i$ is not among the subscripts of $C[{>}i]e_{1,\dots,k}$) or 0 (since $r$ is even, see proof of Proposition~\ref{ejd}). Hence, $j'\neq -i$, otherwise the subscripts of $C[{>}(i-1)]e_{1,\dots,k}$ would contain 0 or $-i$ contradicting~\eqref{samesubs}. Therefore, $d'_{i,j'}=1$ and $E'=E_{j',i}$, we are to show that $j'=j$.


We see that $j'$ is among the subscripts of $E'C'[{>}i]e_{1,\dots,k}$. Furthermore, $C[{<}i]$ contains no factors of the form $E_{a,j'}$. Indeed, $E^{J,D}_{i'}=E_{w^{O,J}(i'),i'}$ for $i'\in(i_{r+1},i)$ while $E^{J,D}_{i'}$ with $i'\le i_{r+1}$ only contain $E_{a,b}$ with $|a|,|b|\le i_{r+1}<i<|j'|$. Thus, $j'$ must be among the subscripts of $e_{J,D}$. In view of Proposition~\ref{ejd}, to complete the case of even $r$ we may show that $|j'|\in{}]i_{r+1},i_r[$ and $j'$ is not equal to $w^{O,J}(i')$ for any $i'\in]i_{r+1},i_r{[}\bs\{i\}$.

Consider $a$, $a'$ such that $r(i,a)=j$ and $r(i,a')=j'$. Proposition~\ref{wrproperties}(a) describes $a$ as the $\lessdot$-maximal element for which $(i,a)$ is contained in $M_O(J)$. Denote the $j'$th pipe of the set $O$ by $\cP$, it contains $(i,a')$ by Proposition~\ref{wrproperties}(d). Recall that $(i,j')\le(i,j)$ (the order in Definition~\ref{monorder}), let us show that if $(i,j')<(i,j)$, then $\cP$ contains an element $(i,a'')\notin J$. Indeed, if $(i,a')\notin J$, we set $a''=a'$. If $(i,j')<(i,j)$ but $(i,a')\in J$, then we cannot have $(i,a)\in O$: otherwise $(i,a)$ would be $\prec$-maximal among elements of the form $(i,b)$ contained in $J\cap O$, hence $(i,j)$ would be $<$-minimal among all $(i,r(i,b'))$ with $(i,b')\in J$. This leaves only one possibility: $(i,a)\in(\max_\prec J)\bs O$ and $(i,a')$ is $\prec$-maximal among elements of the form $(i,b)$ contained in $J\cap O$. However, this lets us choose $a''$ as the element covering $a$ in the order $\lessdot$, the element $(i,a'')$ will lie in $\cP$ because we have $(i,b)\notin O$ for all $a'\lessdot b\lessdot a''$.

Let the pipe $\cP$ start with $(p_1,q_1),\dots,(p_m,q_m)$ where $(p_1,q_1)=(|j'|,-|j'|)$ and $(p_m,q_m)=(i,a'')$. Note that all $(p_l,q_l)\notin J$, however, all elements of the forms $(i_r,b)$ and $(i_{r+1},b)$ lie in $J$. Since $p_m=i\in{}]i_{r+1},i_r[$ and $(p_{l-1},q_{l-1})$ covers $(p_l,q_l)$, we deduce that all $p_l\in{}]i_{r+1},i_r[$, thus, $|j'|\in{} ]i_{r+1},i_r[$. 

Now, for $i'\in{}]i,i_r[$ we have $w^{O,J}(i')=r(i',b)$ for a certain $(i',b)\in M_O(J)$. Since $i'>i$ and $(i',b)\in J$, we see that $(i',b)$ is not in $\cP$ and hence $w^{O,J}(i')=r(i',b)\neq j'$. Also, no $i'\in{}]i_{r+1},i[$ can satisfy $w^{O,J}(i')= j'$, otherwise $C'$ contains the subword \[E_{w^{O,J}(i'),i'}\dots E_{w^{O,J}(i-1),i-1}E_{j',i}=E_{j',i'}\dots E_{w^{O,J}(i-1),i-1}E_{j',i}\] which acts trivially on $\wedge^k V$. This shows that $j'=j$ and completes the case of even $r$.

The case of odd $r$ is similar. In this case $E^{J,D}_i=E_{-i,-j}$ and the subscripts of $E_{-i,-j}C[{>}i]e_{1,\dots,k}$ include $i$ and 0. This again implies $j'\neq-i$ and $d'_{i,j'}=1$, one shows that $j'=j$. It is checked that $-j'$ is not among the subscripts of $e_{J,D}$ and it remains to show that $|j'|\in{}]i_{r+1},i_r[$ and $-j'$ is not equal to $-w^{O,J}(i')$ for any $i'\in ]i_{r+1},i_r[$. Here the argument is the same as for even $r$, the only difference being that in the last step one considers the subword $E_{-i',-w^{O,J}(i')}\dots E_{-(i-1),-w^{O,J}(i-1)} E_{-i,-j'}$.
\end{proof}


\begin{cor}\label{cor:omn}
The set of lattice points in $\Pi^\rB_O(\om_k)$ with $k\le n-1$ (resp.\ in $\Pi^\rB_O(2\om_n)$) is the set of essential signatures for $\om_k$ (resp.\ $2\om_n$).
\end{cor}

As already mentioned, the following fact is also immediate from Proposition~\ref{baseB}.
\begin{cor}\label{bijection}
Every nonzero multivector $e_{i_1,\dots,i_k}$ is equal to $\pm e_{J,D}$ for exactly one pair $J,D$ with $J\in\cJ_k$.
\end{cor}

The final ingredient needed to prove Theorem~\ref{essentialB} is

\begin{proposition}\label{prop:essentialSpinor}
The set of lattice points in $\Pi^\rB_O(\om_n)$ is the set of essential signatures for $\om_n$.
\end{proposition}
\begin{proof}
A description of the spin representation $V_{\om_n}$ can be found in~\cite[Section 13.5]{carter}. It is spanned by vectors $v_D$ with $D$ ranging over subsets of $[1,n]$. The vector $v_D$ has weight $\om_n-\sum_{i\in D}\varepsilon_i$ and we assume $v_{\varnothing}=v_{\om_n}$. Furthermore, if $i\notin D$, then $f_{i,-i}v_D$ is a nonzero multiple of $v_{D\cup\{i\}}$.

By Proposition~\ref{spinpoints} the set $\Pi^\rB_O(\om_n)\cap\bZ^{P\bs A}$ consists of all $y^D$ with $D\subset[1,n]$. By the above, $f^{y^D}v_{\om_n}$ is a nonzero multiple of $v_D$. To show that the signatures $y^D$ are essential we check that if  a PBW monomial $f^d$ has $\fg$-weight $-\sum_{i\in D}\varepsilon_i$, then $d\ge y^D$. Indeed, suppose $d\neq y_D$ and consider the $<$-minimal $(i,j)$ such that $d_{i,j}\neq (y_D)_{i,j}$. If $j\neq -i$, then $d(i)\ge 1>y_D(i)$ and we immediately have $d>y^D$. If $j=-i$, then $d(i)>y^D(i)$ unless $i\in D$ and $d_{i,-i}=0$. But in this case the $i$th coordinate with respect to the basis $\varepsilon_1,\dots,\varepsilon_n$ of the $\fg$-weight of $f^d$ will be $1/2$ rather than $-1/2$.
\end{proof}

\begin{proof}[Proof of Theorem~\ref{essentialB}]
In Corollary~\ref{cor:omn} and Proposition~\ref{prop:essentialSpinor} we have established the claim for $\la=\om_i$ and $\la=2\om_n$. We proceed by induction on (the sum of coordinates of) $\la$. Consider $\la$ which is not fundamental and not $2\om_n$. Then $\la$ can be written as $\la=\mu+\nu$ where $\mu=(a_1,\dots,a_n)$ and $\nu=(b_1,\dots,b_n)$ are nonzero integral dominant weights and at least one of $a_n$ and $b_n$ is even. By Proposition~\ref{minkowskiB} we have 
\begin{equation}\label{inductionstep}
\Pi^\rB_O(\la)\cap\bZ^{P\bs A}=\Pi^\rB_O(\mu)\cap\bZ^{P\bs A}+\Pi^\rB_O(\nu)\cap\bZ^{P\bs A}.    
\end{equation}
By the induction hypothesis and Proposition~\ref{semigroup} all elements of the right-hand side of~\eqref{inductionstep} are essential signatures for $\la$. However, by Proposition~\ref{pointsdim} the left-hand side of~\eqref{inductionstep} contains exactly $\dim V_\la$ points which completes the step.
\end{proof}




\subsection{Toric degenerations and Newton--Okounkov bodies}

We show how the constructed monomial bases can be utilized to obtain toric degenerations and Newton--Okounkov bodies of flag varieties. The approach in this subsection can be viewed as an extension of the methods in~\cite{FFL3}.

Let $G$ denote the Lie group $SO_{2n+1}(\bC)$. Fix an integral dominant $\la=(a_1,\dots,a_n)$ with $a_n$ even. Consider the projectivization $\bP(V_\la)$ and let $[v_\la]\in\bP(V_\la)$ denote the class of $v_\la$. The orbit $F_\la=G[v_\la]$ is the partial flag variety associated with $\la$. 

\begin{remark}
By considering $\mathrm{Spin}_{2n+1}(\bC)$ instead of $SO_{2n+1}(\bC)$ we can define $F_\la$ for odd $a_n$ as well. However, $F_\la$ is determined up to isomorphism by the set of those $k$ for which $a_k\neq 0$ and so is the toric variety of $\cQ^\rB_O(\la)$. Therefore, the assumption that $a_n$ is even does not make (the second claim of) Theorem~\ref{degenmainB} less general.  
\end{remark}

For $d\in\bZ_{\ge 0}^{P\bs A}$ denote $K_d=\prod_{i,j}(d_{i,j}!)$. By Theorem~\ref{essentialB} we have a basis in $V_\la$ consisting of the vectors $v_y=f^yv_\la/K_y$ with $y\in\Pi^\rB_O(\la)\cap\bZ^{P\bs A}$. This basis induces homogeneous coordinates on $\bP(V_\la)$ and identifies its homogeneous coordinate ring with $S=\bC[X_y]_{y\in\Pi^\rB_O(\la)\cap\bZ^{P\bs A}}$. The subvariety $F_\la$ is cut out by an ideal $I\subset S$. With respect to the standard $\bZ$-grading the $m$th homogeneous component of $S/I$ is identified with $H^0(F_\la,\cO_{F_\la}(m))$ and has dimension $\dim V_{m\la}$ by the Borel--Weil theorem.

The toric variety of $\Pi^\rB_O(\la)$ (and hence of $\cQ^\rB_O(\la)$) is also cut out by an ideal $I_O\subset S$. This ideal is the kernel of the map $\varphi_O$ to the ring $\bC[P\bs A,t]=\bC[t][z_{i,j}]_{(i,j)\in P\bs A}$ given by \[\varphi_O(X_y)=tz^y=t\prod_{(i,j)\in P\bs A} z_{i,j}^{y_{i,j}}.\]
The zero set of $I_O$ coincides with the toric variety of $\Pi^\rB_O(\la)$ because the polytope $\Pi^\rB_O(\la)$ is normal, the latter follows from Lemma~\ref{pipointdecompB}.

Now consider the exponential map $\exp:\fg\to G$ and the map $\theta:\bC^{P\bs A}\to G$ given by \[\theta((c_{i,j})_{(i,j)\in P\bs A})=\prod \exp(c_{i,j}f_{i,j})\] with the factors ordered as usual first by $i$ increasing from left to right and then by $j$. Let $U_-\subset G$ denote the unipotent subgroup tangent to $\fn_-=\bigoplus_{i,j}\bC f_{i,j}$. We will use the following standard fact (see, for instance,~\cite[Proposition 8.2.1]{Sp}).
\begin{proposition}\label{imageopen}
The map $\theta$ is an isomorphism (of varieties) between $\bC^{P\bs A}$ and $U_-$.  
\end{proposition}

\begin{theorem}\label{degenmainB}
$I_O$ is an initial ideal of $I$ and, consequently, the toric variety of $\cQ^\rB_O(\la)$ is a flat degeneration of $F_\la$.
\end{theorem}
\begin{proof}
We have the following map between affine spaces:
\[\theta(\text{-}) v_\la:\mathbb{C}^{P \bs A}\rightarrow V_\lambda.\]
Note that this map is polynomial because the elements $f_{i,j}$ act nilpotently. 
In other words, we have polynomials $p_y\in\bC[z_{i,j}]_{(i,j)\in P\bs A}$ indexed by $y\in\Pi^\rB_O(\la)\cap\bZ^{P\bs A}$ such that the coordinates of the vector $\theta((c_{i,j}))v_\la$ in the basis $\{v_y\}$ are the values $p_y|_{z_{i,j}=c_{i,j}}$. These values are also the homogeneous coordinates of the point $\theta((c_{i,j}))[v_\la]\in\bP(V_\la)$. By Proposition~\ref{imageopen} we have $\theta(\bC^{P\bs A})[v_\la]=U_-[v_\la]$, the latter orbit is open in $F_\la$ (it is the open Schubert cell). 
Hence, $I$ is the kernel of the map $\varphi:X_y\mapsto tp_y$ from $S$ to $\bC[P\bs A,t]$. 

The polynomial $p_y$ will be a linear combination of those monomials $z^d$ for which the coordinate of $f^dv_\la$ corresponding to $v_y$ is nonzero. Every such monomial will occur with a nonzero coefficient. Any such $d$ satisfies $d\ge y$ because the signature $y$ is essential for $\la$ (Theorem~\ref{essentialB}). Also note that the monomial $z^y$ appears in $p_y$ with coefficient 1 due to our choice of the value $K_y$. This means that $\initial_{>}p_y=z^y$ where $>$ (note the direction) is naturally viewed as a monomial order on $\bC[z_{i,j}]_{(i,j)\in P\bs A}$. If we extend $>$ to a total monomial order on $\bC[P\bs A,t]$ by setting $t^{m_1}z^{d_1}>t^{m_2}z^{d_2}$ when $m_1>m_2$ or ($m_1=m_2$ and $d_1>d_2$), we have $\initial_> tp_y=tz^y$.

We see that $\varphi_O(X_y)=\initial_>\varphi(X_y)$ so in the notations of Proposition~\ref{idealsubalg} we have $\varphi_O=\varphi_>$. The proposition will imply that $I_O=\ker\varphi_O$ is an initial ideal of $I=\ker\varphi$ if we show that $\varphi_O(S)=\initial_>\varphi(S)$, i.e.\ that the elements $tp_y$ form a sagbi basis. By construction we have $\varphi_O(S)\subset\initial_>\varphi(S)$. Furthermore, both $\varphi_O(S)$ and $\varphi(S)$ are homogeneous with respect to degree in $t$ and have an $m$th homogeneous component of dimension \[|\Pi^\rB_O(m\la)\cap\bZ^{P\bs A}|=\dim V_{m\la}.\] Therefore, $\varphi_O(S)$ and $\initial_>\varphi(S)$ have equal graded dimensions and hence coincide.
\end{proof}

For the remainder of this subsection we assume that $\la$ is regular so that $F_\la$ is the complete flag variety. Consider the line bundle $\cL=\cO_{F_\la}(1)$, i.e.\ the $G$-equivariant line bundle associated with the weight $\la$. In the proof of Theorem~\ref{degenmainB} we have identified $H^0(F_\la,\cL)$ with the image of the degree 1 homogenous component of $S$ under $\varphi$. This is the subspace in $\bC[P\bs A,t]$ spanned by the polynomials $tp_y$, $y\in\Pi^\rB_O(\la)\cap\bZ^{P\bs A}$. Choose $\tau\in H^0(F_\la,\cL)$ as $\tau=tp_0=t$.

By Proposition~\ref{imageopen} and since $\la$ is regular, the map $c\mapsto\theta(c)[v_\la]$ is a bijection from the space $\bC^{P\bs A}$ to the open Schubert cell $U_-[v_\la]$. This provides a birational equivalence between $\bC^{P\bs A}$ and $F_\la$ and lets us identify $\bC(F_\la)$ with the field $\mathbb K=\bC(z_{i,j})_{(i,j)\in P\bs A}$. Here $h\in\mathbb K$ is identified with the function taking value $h|_{z_{i,j}=c_{i,j}}$ at the point $\theta((c_{i,j}))[v_\la]$. The \textit{highest term valuation} on $\mathbb K$ given by the monomial order $>$ is a map $\nu:\mathbb K\bs\{0\}\to\bZ^{P\bs A}$ defined as follows. For $p\in\bC[z_{i,j}]_{(i,j)\in P\bs A}$ one sets $z^{\nu(p)}=\initial_> p$ and extends $\nu$ to $\mathbb K$ so that $\nu(gh)=\nu(g)+\nu(h)$. This is a $(\bZ^{P\bs A},>)$-valuation (defined as in Definition~\ref{valdef}). Similarly to Subsection~\ref{nosec} we have

\begin{definition}
The Newton--Okounkov body of $F_\la$ associated with $\cL$, $\tau$ and $\nu$ is the convex hull closure \[\Delta=\overline{\conv\left\{\left.\frac{\nu(\sigma/\tau^{\otimes m})}m\right|m\in\bZ_{>0},\sigma\in H^0(F_\la,\cL^{\otimes m})\bs\{0\}\right\}}\subset\bR^{P\bs A}.\]
\end{definition}

\begin{theorem}
$\Delta=\Pi^\rB_O(\la)$.
\end{theorem}
\begin{proof}
Since $\cL^{\otimes m}=\cO_{F_\la}(m)$, we have identified $H^0(F_\la,\cL^{\otimes m})$ with $\varphi(S[m])$ where $S[m]$ is the degree $m$ component in $S$. This provides an isomorphism between the ring $\bigoplus_m H^0(F_\la,\cL^{\otimes m})$ and $\varphi(S)$. In particular the space $H^0(F_\la,\cL^{\otimes m})/\tau^{\otimes m}\subset\mathbb K$ is precisely $\varphi(S[m])/t^m$. The latter space is spanned by all products $p_{y_1}\dots p_{y_m}$ where $y_1,\dots,y_m$ are points in $\Pi^\rB_O(\la)\cap\bZ^{P\bs A}$. Note that $\nu(p_{y_1}\dots p_{y_m})=y_1+\dots+y_m$. However, by Proposition~\ref{pipointdecompB} all lattice points in $\Pi^\rB_O(m\la)$ have the form $y_1+\dots+y_m$. Thus, 
\begin{equation}\label{valofcomp}
\Pi^\rB_O(m\la)\cap\bZ^{P\bs A}\subset\nu\left(\frac{H^0(F_\la,\cL^{\otimes m})}{\tau^{\otimes m}}\bs\{0\}\right).
\end{equation}
However, the left-hand side of~\eqref{valofcomp} has cardinality $\dim V_{m\la}$ while the right-hand side has cardinality no greater than  $\dim H^0(F_\la,\cL^{\otimes m})=\dim V_{m\la}$, therefore, the two sides coincide. We deduce that for any $m$: \[\conv\left\{\left.\frac{\nu(\sigma/\tau^{\otimes m})}m\right|\sigma\in H^0(F_\la,\cL^{\otimes m})\bs\{0\}\right\}=\frac{\conv(\Pi^\rB_O(m\la)\cap\bZ^{P\bs A})}m=\Pi^\rB_O(\la).\qedhere\]
\end{proof}

Let us now define a total order $\widehat>$ on $\bZ^P$ as follows. For distinct $x,x'\in\bZ^P$ consider the minimal $i$ such that $\sum_jx_{i,j}\neq\sum_jx'_{i,j}$ and set $x\widehat>x'$ if $\sum_jx_{i,j}>\sum_jx'_{i,j}$. If no such $i$ exists, set $x\widehat>x'$ if $\pi(x)>\pi(x')$. Next, define a linear map $\rho:\bR^{P\bs A}\to\bR^P$ so that for $y\in\bZ^{P\bs A}$ one has $\rho(y)_{i,i}=-\sum_jy_{i,j}$ and $\rho(y)_{i,j}=y_{i,j}$ when $j\neq i$. Of course, $\pi\rho(y)=y$ and $\rho(y)\widehat>\rho(y')$ if and only if $y>y'$. 

We next define another total order $>^\xi$ on $\bZ^P$ by setting $x>^\xi x'$ if $\xi(x)\widehat>\xi(x')$. We also define a map $\nu^\xi:\mathbb K\bs\{0\}\to\bZ^P$ by setting $\nu^\xi(h)=\xi^{-1}\rho(\nu(h))$. Simply by tracing the definitions one checks that this is a $(\bZ^P,>^\xi)$-valuation. The corresponding Newton--Okounkov body is \[\Delta^\xi=\overline{\conv\left\{\left.\frac{\nu^\xi(\sigma/\tau^{\otimes m})}m\right|m\in\bZ_{>0},\sigma\in H^0(F_\la,\cL^{\otimes m})\bs\{0\}\right\}}\subset\bR^P.\]

Since $0\in\Pi^\rB_O(\la)$, we have a unique point $x_\la\in\cQ^\rB_O(\la)\cap\bZ^P$ such that $\pi\xi(x_\la)=0$.
\begin{theorem}\label{mainNOB}
$\Delta^\xi=\cQ^\rB_O(\la)-x_\la$.    
\end{theorem}
\begin{proof}
One sees that $\xi(x_\la)_{i,i}=\la(i)$ while all other $\xi(x_\la)_{i,j}=0$. This implies
\begin{equation}\label{liftedpi}
\rho(\Pi^\rB_O(\la))=\xi(\cQ^\rB_O(\la))-\xi(x_\la)
\end{equation}
since $\rho(\Pi^\rB_O(\la))$ consists of $x$ such that $\pi(x)\in\Pi^\rB_O(\la)$ and all $\sum_jx_{i,j}=0$ while $\xi(\cQ^\rB_O(\la))$ consists of $x$ such that $\pi(x)\in\Pi^\rB_O(\la)$ and $\sum_jx_{i,j}=\la(i)$. Applying~\eqref{liftedpi} we compute \[\Delta^\xi=\xi^{-1}\rho(\Delta)=\xi^{-1}\rho(\Pi^\rB_O(\la))=\cQ^\rB_O(\la)-x_\la.\qedhere\]
\end{proof}

\begin{remark}\label{tableauxB}
We conclude this section with a brief discussion of Young tableaux. Since we do not work with Pl\"ucker coordinates on the type B flag variety, we do not obtain standard monomial theories and an analog of Corollary~\ref{standmon}. However, the above still provides a certain notion of standard tableaux which can, for instance, be used to write character formulas. Consider $\la=(a_1,\dots,a_n)$ with $a_n$ even. For a lattice point $y\in\Pi^\rB_O(\la)$ consider the decomposition given by Lemma~\ref{pipointdecompB}(a). We may encode this decomposition by a tableau $T_y$ whose $(m+1-i)$th column contains the elements $w^{O,J_i}(1),\dots,w^{O,J_i}(k)$ where $J_i\in\cJ_k$. In addition, the $i$th element in the first column is \textit{marked} if $i\in D$. We may declare tableaux of the form $T_y$ to be standard. This family consists of tableaux obtained for the type C highest weight $(a_1,\dots,a_n/2)$ according to Remark~\ref{tableaux} where, in addition, the $i$th element of the first column can be marked if it is equal to $-i$. For $O=P\bs B$ we obtain the Koike--Terada $SO(2n+1)$-tableaux of~\cite{KT}. For $O=A$ we obtain a new family of type B tableaux which are the symplectic PBW-semistandard tableaux of~\cite{Ba} with marked elements in the first  column.

We may now define the weight $\mu(T)$ of a tableau $T$ to be the sum of $\sgn(a)\varepsilon_{|a|}$ over all non-marked elements $a$ in $T$. Theorem~\ref{essentialB} then implies that the character of $V_\la$ is the sum of $e^{\mu(T)}$ over all standard tableaux $T$, similarly to the classical theory.
\end{remark}

\bibliographystyle{plainurl}
\bibliography{refs.bib}

\end{document}